\documentclass{amsart}

\addtocontents{toc}{\setcounter{tocdepth}{1}}
\usepackage{amsmath}
\usepackage{amsthm}
\usepackage{amssymb}
\usepackage[matrix,arrow]{xy}

\numberwithin{equation}{section}
\newtheorem{theorem}[equation]{Theorem}
\newtheorem{lemma}[equation]{Lemma}
\newtheorem{proposition}[equation]{Proposition}
\newtheorem{corollary}[equation]{Corollary}
\newtheorem{maintheorem}{Theorem}
 
\newtheorem{maincorollary}[maintheorem]{Corollary}
\theoremstyle{remark}
\newtheorem{example}[equation]{Example}
\newtheorem{definition}[equation]{Definition}
\newtheorem*{notation}{Notation}
\newtheorem{remark}[equation]{Remark}

\newcommand{\dual}{\sharp}
\newcommand{\cat}[1]{\mathsf{#1}}
\newcommand{\D}{\cat{D}}
\newcommand{\InfD}{\D^\infty}
\newcommand{\fm}{\mathfrak{m}}
\newcommand{\BB}{\mathcal{B}}
\newcommand{\ALOWB}{\Aus_{\BB}}
\newcommand{\BLOWB}{\Bass_{\BB}}
\newcommand{\AUPB}{\Aus^{\BB}}
\newcommand{\BUPB}{\Bass^{\BB}}
\newcommand{\Aus}{\mathbb{A}}
\newcommand{\Bass}{\mathbb{B}}

\newcommand{\Z}{\mathbb{Z}}
\newcommand{\CC}{\mathcal{C}}
\renewcommand{\CC}{\cat{C}}
\newcommand{\DD}{\mathcal{D}}
\renewcommand{\DD}{\cat{D}}
\newcommand{\baraug}[1]{B^{\infty}_{\aug}#1}
\newcommand{\barinfty}[1]{B^{\infty}#1}
\newcommand{\koszul}[1]{E(#1)}
\newcommand{\triang}[2]{\Triang_{#1} (#2)}
\newcommand{\thick}[2]{\Thick_{#1} (#2)}
\newcommand{\trinf}[2]{\Triang_{#1}^{\infty} (#2)}
\newcommand{\thinf}[2]{\Thick_{#1}^{\infty} (#2)}
\newcommand{\loc}[2]{\Loc_{#1}(#2)}
\newcommand{\locinf}[2]{\Loc_{#1}^{\infty}(#2)}
\newcommand{\env}[1]{U\!#1}
\newcommand{\ltensor}[1]{\underset{#1}{\overset{L}{\otimes}}}
\newcommand{\hominfty}[1]{\overset{\infty}{\mathsf{Hom}}{_{#1}^{\bullet}}}
\newcommand{\tensorinfty}[1]{\underset{#1}{\overset{\infty}{\otimes}}}
\newcommand{\Mod}{\cat{Mod}}
\newcommand{\Ho}{\cat{Ho}}
\newcommand{\Hocat}{\Ho \Aug}
\newcommand{\Triang}{\cat{triang}}
\newcommand{\Thick}{\cat{thick}}
\newcommand{\Loc}{\cat{loc}}
\newcommand{\Alg}{\cat{Alg}}
\newcommand{\Aug}{\ainfty_{\aug}}
\newcommand{\ainfty}{\Alg^{\infty}}
\newcommand{\DGC}{\cat{DGC}}
\newcommand{\DGA}{\cat{DGA}}
\newcommand{\id}{\mathbf{1}}

\DeclareMathOperator{\injdim}{injdim}
\DeclareMathOperator{\Ext}{Ext}
\DeclareMathOperator{\End}{End}
\DeclareMathOperator{\RHom}{RHom}
\DeclareMathOperator{\proj}{proj}
\DeclareMathOperator{\Proj}{Proj}
\DeclareMathOperator{\Hom}{Hom}
\DeclareMathOperator{\aug}{aug}
\DeclareMathOperator{\coaug}{coaug}
\DeclareMathOperator{\cok}{cok}
\DeclareMathOperator{\op}{op}
\DeclareMathOperator{\per}{per}
\DeclareMathOperator{\fg}{fg}
\DeclareMathOperator{\fd}{fd}
\DeclareMathOperator{\sm}{sm}

\begin{document}

\begin{center} {\Large \bf Koszul Equivalences in
$A_\infty$-Algebras\\}

\bigskip\bigskip
\renewcommand{\thefootnote}{{}}

 D.-M. Lu $^a$,
 J. H. Palmieri $^b$, Q.-S. Wu $^c$ and J. J. Zhang $^b$ \footnote{{\it E-mail addresses:} dmlu@zju.edu.cn (D.-M. Lu),
 palmieri@math.washington.edu (J. H. Palmieri), qswu@fudan.edu.cn (Q.-S. Wu), zhang@math.washington.edu (J. J. Zhang).}\\

\bigskip {\footnotesize \it $^a$ Department of Mathematics, Zhejiang
University, Hangzhou 310027, China \\ $^b$ Department of
Mathematics, Box 354350, University of Washington, Seattle, WA
98195, USA\\
$^c$ Institute of Mathematics, Fudan University, Shanghai, 200433,
China}
\end{center}

\bigskip

\begin{center}
\begin{minipage}{100mm}
\bigskip

{\textbf{Abstract}}
\medskip

\footnotesize {We prove a version of Koszul duality and the induced
derived equivalence for Adams connected $A_\infty$-algebras that
generalizes the classical Beilinson-Ginzburg-Soergel Koszul duality.
As an immediate consequence, we give a version of the
Bern\v{s}te{\u\i}n-Gel'fand-Gel'fand correspondence for Adams
connected $A_\infty$-algebras.

We give various applications. For example, a connected graded
algebra $A$ is Artin-Schelter regular if and only if its Ext-algebra
$\Ext^\ast_A(k,k)$ is Frobenius. This generalizes a result of Smith
in the Koszul case. If $A$ is Koszul and if both $A$ and its Koszul
dual $A^!$ are noetherian satisfying a polynomial identity, then $A$
is Gorenstein if and only if $A^!$ is. The last statement implies
that a certain Calabi-Yau property is preserved under Koszul
duality.

\bigskip

\noindent\textit{MSC}: 16A03,16A62,16E65

\bigskip

\noindent\textit{Keywords.} {$A_\infty$-algebra, graded algebra,
Artin-Schelter regular algebra, Koszul duality, derived equivalence,
Gorenstein property}}

\end{minipage}
\end{center}

\bigskip

\section*{Introduction}

Koszul duality is an incredibly powerful tool used in many
areas of mathematics. One aim of this paper is to unify
some generalizations by using $A_\infty$-algebras.
Our version is comprehensive enough to recover the original
version of Koszul duality and the induced derived equivalences
due to Beilinson-Ginzburg-Soergel \cite{BGSo} and most of the
generalizations in ring theory and algebraic geometry.
Although we will restrict ourselves to Adams connected
$A_\infty$-algebras (a natural extension of a connected graded
algebras -- see Definition~\ref{xxdefn2.1}), we have set up
a framework that will work for other classes of algebras
arising from representation theory and algebraic geometry.

We fix a commutative field $k$ and work throughout with vector
spaces over $k$.  We define $A_\infty$-algebras over $k$ in
Definition~\ref{xxdefn1.1}.

Similar to \cite[Section 11]{LP04} we define the Koszul dual of
an $A_\infty$-algebra $A$ to be the vector space dual of the bar
construction of $A$ -- see Section \ref{xxsec2} for details. This idea
is not new and dates back at least
to Beilinson-Ginsburg-Schechtman \cite{BGSc} for graded algebras.
Keller also took this approach in \cite{Ke94} for differential
graded algebras. Our first result is a generalization of
\cite[Theorem 2.10.2]{BGSo}.

\begin{maintheorem}
\label{xxthmA} Let $A$ be an augmented $A_\infty$-algebra. Suppose
that the Koszul dual of $A$ is locally finite. Then the
double Koszul dual of $A$ is $A_\infty$-isomorphic to $A$.
\end{maintheorem}

This is proved as Theorem \ref{xxthm2.4}.
A special case of the above theorem was proved in
\cite[Theorem 11.2]{LP04}.

As in \cite{BGSo} we prove several versions of equivalences of
derived categories induced by the Koszul duality. Let $\InfD(A)$
be the derived category of right $A_\infty$-modules over $A$.
Let  $\InfD_{\per}(A)$ (respectively, $\InfD_{\fd}(A)$) denote the
full triangulated subcategory of $\InfD(A)$ generated by all
perfect complexes (respectively, all right $A_\infty$-modules whose
homology is finite-dimensional) over $A$. The next result
is a generalization of \cite[Theorem 2.12.6]{BGSo}.

\begin{maintheorem}
\label{xxthmB} Let $A$ be an Adams connected $A_\infty$-algebra
and $E$ its Koszul dual. If $HE$ is finite-dimensional,
then there is an equivalence of triangulated categories
$\InfD_{\per}(A)\cong \InfD_{\fd}(E)$.
\end{maintheorem}

This is proved as Corollary \ref{xxcor7.2}(b).

Other equivalences of triangulated categories can be found in Sections
\ref{xxsec4} and \ref{xxsec5}. If $A$ is either Artin-Schelter regular
(Definitions \ref{xxdefn9.1}(c) and \ref{xxdefn9.2}(c)) or right
noetherian with finite global dimension, then $HE$ is
finite-dimensional and hence Theorem \ref{xxthmB} applies.

Koszul duality has many applications in ring theory, representation
theory, algebraic geometry, and other areas.  The next result is a
generalization of the Bern\v{s}te{\u\i}n-Gel'fand-Gel'fand
correspondence that follows from Theorem \ref{xxthmB}.
Let $\underline{\InfD_{\fg}}(A)$ be the stable derived category of
$A_\infty$-modules over $A$ whose homology is finitely generated
over $HA$, and let $\InfD(\proj A)$ be the derived category of the
projective scheme of $A$.  These categories are defined in Section
\ref{xxsec10}, and the following theorem is part of Theorem
\ref{xxthm10.2}.

\begin{maintheorem}
\label{xxthmC} Let $A$ be an Adams connected $A_\infty$-algebra
that is noetherian Artin-Schelter regular. Let $E$ be the Koszul dual
of $A$. Then $HE$ is finite-dimensional and there is an
equivalence of triangulated categories
$\InfD(\proj A)\cong \underline{\InfD_{\fg}}(E)$.
\end{maintheorem}

Applications of Koszul duality in ring theory are surprising and
useful. We will mention a few results that are related to the
Gorenstein property. In the rest of this introduction we let $R$ be a
connected graded associative algebra over a base field $k$.

\begin{maincorollary}
\label{xxcorD}
Let $R$ be a connected graded algebra. Then $R$
is Artin-Schelter regular if and only if the {\rm Ext}-algebra
$\bigoplus_{i\in {\mathbb Z}}\Ext^i_R(k_R,k_R)$ is Frobenius.
\end{maincorollary}

This result generalizes a theorem of Smith \cite[Theorem 4.3 and
Proposition 5.10]{Sm} that was proved for Koszul algebras.  It is
proved in Section \ref{xxsec9.3}.  Corollary \ref{xxcorD} is a
fundamental result and the project \cite{LP07} was based on it.

The Gorenstein property plays an important role in commutative
algebra and algebraic geometry. We prove that the Gorenstein
property is preserved under Koszul duality; see Section
\ref{xxsec9.4} for details.

\begin{maincorollary}
\label{xxcorE} Let $R$ be a Koszul algebra and let $R^!$ be the
Koszul dual of $R$ in the sense of Beilinson-Ginzburg-Soergel
\cite{BGSo}. If $R$ and $R^!$ are both noetherian having balanced
dualizing complexes, then $R$ is Gorenstein if and only if
$R^!$ is.
\end{maincorollary}

The technical hypothesis about the existence of balanced
dualizing complexes can be checked when the rings are close to
being commutative. For example, Corollary \ref{xxcorE} holds
when $R$ and $R^!$ are noetherian and satisfy a polynomial
identity. This technical hypothesis is presented because we
do not understand noncommutative rings well enough. We do not
know any example in which the technical hypothesis is
necessary; however, Corollary \ref{xxcorE} does fail for
non-noetherian rings -- for example, the free algebra
$R=k\langle x,y \rangle$ is Gorenstein, but $R^!\cong k \langle x,y
\rangle/(x^2,xy,yx,y^2)$ is not.

Note that Koszul duality preserves the Artin-Schelter
condition (Proposition \ref{xxprop9.3}). Under the
technical hypothesis of Corollary \ref{xxcorE}
the Artin-Schelter condition is equivalent to the Gorenstein
property. Therefore Corollary \ref{xxcorE} follows.
We can restate Corollary \ref{xxcorE} for $A_\infty$-algebras
in a way that may be useful for studying the Calabi-Yau property of
the derived category $\InfD(A)$ (see the discussion in Section
\ref{xxsec9.4}).

The following is proved in Section \ref{xxsec9.4}.

\begin{maincorollary}
\label{xxcorF}
Let $A$ be an Adams connected commutative differential graded
algebra such that $\RHom_A(k,A)$ is not quasi-isomorphic to
zero. If the Ext-algebra $\bigoplus_{i\in {\mathbb Z}}
\Ext^i_A(k_A,k_A)$ is noetherian, then $A$ satisfies the
Artin-Schelter condition.
\end{maincorollary}

The hypothesis on $\RHom_A(k,A)$ is a version of finite
depth condition which is very mild in commutative ring theory
and can be checked under some finiteness conditions. This is
automatic if $A$ is a finitely generated associative commutative
algebra. As said before, the Artin-Schelter condition is equivalent
to the Gorenstein property under appropriate hypotheses. Hence
Corollary \ref{xxcorF} relates the Gorenstein property of $R$
with the noetherian property of $R^!$ and partially explains
why Corollary \ref{xxcorE} holds, and at the
same time it suggests that Corollary \ref{xxcorE} should be
a consequence of a more basic statement that relates the
noetherian property of $R$ with the Gorenstein property of
$R^!$. On the other hand, we believe that there should be a
version of Corollary \ref{xxcorE} without using the
noetherian property. As we commented for Corollary \ref{xxcorE},
Corollary \ref{xxcorF} should hold in a class of noncommutative
rings $R$. Corollary \ref{xxcorF} is also a variation of a result
of B{\o}gvad and Halperin about commutative complete intersection
rings \cite{BH}.

This paper is part four of our $A_\infty$-algebra project
and is a sequel to \cite{LP04, LP07, LP08}. Some results were
announced in \cite{LP04}. For example, Theorem \ref{xxthmB}
and Corollary \ref{xxcorD} were stated in \cite{LP04} without proof.
We also give a proof of \cite[Theorem 11.4]{LP04} in Section \ref{xxsec5}.

The paper is divided into three parts: Koszul duality for algebras,
Koszul duality for modules, and applications in ring theory.

Part I consists of Sections \ref{xxsec1} and \ref{xxsec2}.  Section
\ref{xxsec1} gives background material on $A_\infty$-algebras and
their morphisms.  The reader may wish to skim it to see the
conventions and notation used throughout the paper.  Theorem
\ref{xxthmA} is proved in Section \ref{xxsec2}, and we use it to
recover the classical Koszul duality of Beilinson, Ginzburg,
and Soergel.  We also discuss a few examples.

Part II consists of Sections \ref{xxsec3}--\ref{xxsec8}.  Section
\ref{xxsec3} gives background material on $A_\infty$-modules; most of
this is standard, but the results on opposites is new.  Section
\ref{xxsec4} sets up a framework for proving equivalences of various
derived categories of DG modules over a DG algebra.  Section
\ref{xxsec5} uses this framework to prove DG and $A_\infty$ versions
of the results of Beilinson, Ginzburg, and Soergel which establish
equivalences between certain derived categories of modules over a
Koszul algebra and over its Koszul dual.  The point of Section
\ref{xxsec6} is a technical theorem which allows us, in Section
\ref{xxsec7}, to rederive the classical results from the
$A_\infty$-algebra results.  Theorem \ref{xxthmB} is proved in Section
\ref{xxsec7}, also.  We discuss a couple of examples in Section
\ref{xxsec8}.

Part III consists of Sections \ref{xxsec9}--\ref{xxsec10}.  Section
\ref{xxsec9} discusses Artin-Schelter regular algebras and Frobenius
algebras, from the $A_\infty$-algebra point of view, and includes
proofs of Corollaries \ref{xxcorD}, \ref{xxcorE}, and \ref{xxcorF}.
Section \ref{xxsec10} gives an $A_\infty$-version of the BGG
correspondence; Theorem \ref{xxthmC} is proved there.

\tableofcontents

\part{Koszul duality for algebras}

\section{Background on $A_\infty$-algebras}
\label{xxsec1}

In this section, we describe background material necessary for the
rest of the paper.  There are several subsections: grading conventions
and related issues; $A_\infty$-algebras and morphisms between them;
the bar construction; and homotopy for morphisms of
$A_\infty$-algebras.

\subsection{Conventions}
\label{xxsec1.1}

Throughout we fix a commutative base field $k$. Unless otherwise
stated, every chain complex, vector space, or algebra will be over
$k$.  The unadorned tensor product $\otimes$ is over $k$ also.

Vector spaces (and the like) under consideration in this paper are
bigraded, and for any bihomogeneous element $a$, we write $\deg a =
(\deg_{1}(a), \deg_{2}(a))\in {\mathbb Z}\times G$ for some abelian
group $G$. The second grading is called the \emph{Adams grading}. In
the classical setting $G$ is trivial, but in this paper we have
$G={\mathbb Z}$; many of the abstract assertions in this paper hold
for any abelian group $G$.  If $V$ is a bigraded vector space, then
the degree $(i,j)$ component of $V$ is denoted by $V^i_j$. Usually we
work with bihomogeneous elements, with the possibility of ignoring the
second grading. All chain complexes will have a differential of degree
$(1,0)$. The Koszul sign convention is in force throughout the paper,
but one should ignore the second grading when using it: when
interchanging elements of degree $(i,s)$ and $(j,t)$, multiply by
$(-1)^{ij}$.

Given a bigraded vector space $V$, we write $V^{\dual}$ for its
graded dual.  Its \emph{suspension} $SV$ is the bigraded space with
$(SV)^{i}_{j} = V^{i+1}_{j}$: suspension shifts the first grading down
by one, and ignores the second grading.  Write $s: V \rightarrow SV$
for the obvious map of degree $-1$.  If $V$ has a differential $d_{V}$,
then define a differential $d_{SV}$ on $SV$ by $d_{SV}(sv) =
-sd_{V}(v)$. The \emph{Adams shift} of $V$ is $\Sigma V$ with
$(\Sigma V)^{i}_{j}=V^{i}_{j+1}$.  If $V$ has a differential, then
there is no sign in the differential for $\Sigma V$: $d_{\Sigma V}$ is
just the shift of $d_{V}$.

If $(M,d)$ and $(N,d)$ are complexes, so are $\Hom_{k}(M,N)$ and
$M\otimes N(=M \otimes_{k} N)$, with differentials given by
\begin{gather*}
d(f) = df - (-1)^{\deg_{1} (f)}fd, \quad \forall\; f\in \Hom_k(M,N);\\
d(m \otimes n) = dm \otimes n + (-1)^{\deg_{1} (m)} m \otimes dn,
\quad \forall\;  m\otimes n\in M\otimes N,
\end{gather*}
respectively.

If $\CC$ is a category, we write $\CC (X,Y)$ for morphisms in $\CC$
from $X$ to $Y$.  We reserve $\Hom$ to denote the chain complex with
differential as in the previous paragraph.

\subsection{$A_\infty$-algebras and morphisms}
\label{xxsec1.2}

In this paper, we will frequently work in the category of augmented
$A_\infty$-algebras; in this subsection and the next, we define the
objects and morphisms of this category.  Keller's paper \cite{Ke01}
provides a nice introduction to $A_\infty$-algebras; it also has
references for many of the results which we cite here and in later
subsections.  Lef\`evre-Hasegawa's thesis \cite{Le} provides more
details for a lot of this; although it has not been published, it is
available on-line.  Another reference is \cite{LP04} which contains
some easy examples coming from ring theory. The following definition
is originally due to Stasheff \cite{St}.

\begin{definition}
\label{xxdefn1.1}
An \emph{$A_\infty$-algebra} over $k$ is a $\Z \times \Z$-graded
vector space $A$ endowed with a family of graded $k$-linear maps
\[
m_n: A^{\otimes n} \to A, \quad n\geq 1,
\]
of degree $(2-n,0)$ satisfying the following \emph{Stasheff identities}:
for all $n\geq 1$,
\begin{equation*}
\tag*{\textbf{SI(n)}}
\sum (-1)^{r+st} m_u(\id^{\otimes r}\otimes m_s \otimes \id^{\otimes t})=0,
\end{equation*}
where the sum runs over all decompositions $n=r+s+t$, with $r$,
$t\geq 0$ and $s\geq 1$, and where $u=r+1+t$. Here $\id$ denotes
the identity map of $A$. Note that when these formulas are applied
to elements, additional signs appear due to the Koszul sign rule.

A \emph{DG} (\emph{differential graded}) \emph{algebra} is an
$A_\infty$-algebra with $m_n=0$ for all $n\geq 3$.
\end{definition}

The reader should perhaps be warned that there are several different
sign conventions in the $A_\infty$-algebra literature.  We have
chosen to follow Keller \cite{Ke01}, who is following Getzler and
Jones \cite{GJ}.  Stasheff \cite{St} and Lef\`evre-Hasegawa
\cite{Le} use different signs: they have the sign $(-1)^{rs+t}$ in
\textbf{SI(n)}, and this requires sign changes in other formulas (such
as \textbf{MI(n)} below).

As remarked above, we work with bigraded spaces throughout, and this
requires a (very mild) modification of the standard definitions:
ordinarily, an $A_\infty$-algebra is singly graded and $\deg m_{n} =
2-n$; in our bigraded case, we have put $\deg m_{n} = (2-n,0)$.  Thus
if one wants to work in the singly graded setting, one can just work
with objects concentrated in degrees $(*,0)=\Z\times \{0\}$.

\begin{definition}
\label{xxdefn1.2}
An $A_\infty$-algebra $A$ is \emph{strictly unital} if $A$ contains
an element $1$ which acts as a two-sided identity with respect to
$m_{2}$, and for $n \neq 2$, $m_n(a_1 \otimes \cdots \otimes a_n)=0$
if $a_{i}=1$ for some $i$.
\end{definition}

In this paper we assume that $A_\infty$-algebras (including
DG algebras) are strictly unital.

\begin{definition}
\label{xxdefn1.3}
A \emph{morphism} of $A_\infty$-algebras $f: A \rightarrow B$ is a
family of $k$-linear graded maps
\[
f_n: A^{\otimes n}\to B, \quad n\geq 1,
\]
of degree $(1-n,0)$ satisfying the following \emph{Stasheff morphism
identities}: for all $n\geq 1$,
\begin{equation*}
\tag*{\textbf{MI(n)}}
\sum (-1)^{r+st} f_u(\id^{\otimes r}\otimes m_s\otimes
\id^{\otimes t})=\sum (-1)^{w} m_q(f_{i_1}\otimes f_{i_2}
\otimes \cdots \otimes f_{i_q}),
\end{equation*}
where the first sum runs over all decompositions $n=r+s+t$ with
$r$, $t\geq 0$, $s\geq 1$, and where we put $u=r+1+t$; and the
second sum runs over all $1\leq q\leq n$ and all decompositions
$n=i_1+\cdots + i_{q}$ with all $i_s\geq 1$.  The sign on the
right-hand side is given by
\[
w=(q-1)(i_1 -1) + (q-2) (i_2-1)+ \cdots + 2(i_{q-2}-1)+ (i_{q-1}-1).
\]

When $A_\infty$-algebras have a strict unit (as we usually assume), an
$A_\infty$-morphism between them is also required to be \emph{strictly
unital}, which means that it must satisfy these \emph{unital morphism
conditions}: $f_1(1_A)=1_B$ where $1_A$ and $1_B$ are strict units of
$A$ and $B$ respectively, and $f_n(a_1 \otimes \cdots \otimes a_n)= 0$
if some $a_i=1_A$ and $n\geq 2$ (see \cite[3.5]{Ke01}, \cite[Section
4]{LP04}).
\end{definition}

As with $A_\infty$-algebras, we have modified the grading on
morphisms: we have changed the usual grading of $\deg f_{n} = 1-n$ to
$\deg f_{n} = (1-n,0)$.  The composite of two morphisms is given by a
formula similar to the morphism identities $\mathbf{MI(n)}$; see
\cite{Ke01} or \cite{Le} for details.

\begin{definition}
\label{xxdefn1.4}
A morphism $f: A \rightarrow B$ of $A_\infty$-algebras is
\emph{strict} if $f_{n}=0$ for $n \neq 1$.  The \emph{identity
morphism} is the strict morphism $f$ with $f_{1}=\id$.  A morphism
$f$ is a \emph{quasi-isomorphism} or an \emph{$A_\infty$-isomorphism}
if $f_{1}$ is a quasi-isomorphism of chain complexes.
\end{definition}

Note that quasi-isomorphisms of $A_\infty$-algebras have inverses:
a morphism is a quasi-isomorphism if and only if it is a homotopy
equivalence -- see Theorem~\ref{xxthm1.16} below.

We write $\Alg$ for the category of associative $\Z\times \Z$-graded
algebras
with morphisms being the usual graded algebra morphisms,
and we write $\ainfty$ for the category of $A_\infty$-algebras
with $A_\infty$-morphisms.

Let $A$ and $B$ be associative algebras, and view them as
$A_\infty$-algebras with $m_{n}=0$ when $n \neq 2$.  We point out
that there may be non-strict $A_\infty$-algebra morphisms between
them.  That is, the function
\[
\Alg (A,B) \rightarrow \ainfty (A,B),
\]
sending an algebra map to the corresponding strict
$A_\infty$-morphism, need not be a bijection.  See
Example~\ref{xxex2.8} for an illustration of this.

The following theorem is important and useful.

\begin{theorem}\cite{Ka80}
\label{xxthm1.5}
Let $A$ be an $A_\infty$-algebra and let $HA$ be its cohomology ring.
There is an $A_\infty$-algebra structure on $HA$ with $m_1=0$ and
$m_{2}$ equal to its usual associative product, and with the higher
multiplications constructed from the $A_\infty$-structure of $A$, such
that there is a quasi-isomorphism of $A_\infty$-algebras $HA\to A$
lifting the identity of $HA$.
\end{theorem}

This theorem was originally proved for $\Z$-graded
$A_\infty$-algebras, and holds true in our $\Z\times \Z$-setting.

\subsection{Augmented $A_\infty$-algebras}
\label{xxsec1.3}

A strictly unital $A_\infty$-algebra $A$ comes equipped with a
strict, strictly unital morphism $\eta : k \rightarrow A$.

\begin{definition}
\label{xxdefn1.6}
\begin{enumerate}
\item
A strictly unital $A_\infty$-algebra $A$ is \emph{augmented} if
there is a strictly unital $A_\infty$-algebra morphism $\varepsilon
: A \rightarrow k$ so that $\varepsilon \circ \eta = \id_{k}$.
\item
If $A$ is an augmented $A_\infty$-algebra with augmentation
$\varepsilon : A \rightarrow k$, its \emph{augmentation ideal} is
defined to be $\ker (\varepsilon_{1})$.
\item
A \emph{morphism} of augmented $A_\infty$-algebras
$f: A \rightarrow B$ must be strictly unital and must respect the
augmentations: $\varepsilon_A = \varepsilon_{B} \circ f$.
We write $\Aug$ for the resulting category of augmented
$A_\infty$-algebras.
\end{enumerate}
\end{definition}

\begin{proposition}[Section 3.5 in \cite{Ke01}]
\label{xxprop1.7}
The functor $\Aug \rightarrow \ainfty$ sending an augmented
$A_\infty$-algebra to its augmentation ideal is an equivalence of
categories.  The quasi-inverse sends an $A_\infty$-algebra $A$ to $k
\oplus A$ with the apparent augmentation.
\end{proposition}

Using this equivalence, one can translate results and constructions
for $A_\infty$-algebras to the augmented case.  The bar construction
is an application of this.

\subsection{The bar construction}
\label{xxsec1.4}

The bar construction $B(-)$ is of central importance in this paper,
since we define the Koszul dual of $A$ to be the vector
space dual of its bar construction.  In this subsection, we
describe it.  We also discuss the cobar construction $\Omega (-)$,
the composite $\Omega (B(-))$, and other related issues.

The following definition is a slight variant on that in
\cite[Section 3.6]{Ke01}.

\begin{definition}
\label{xxdefn1.8}
Let $A$ be an
augmented $A_\infty$-algebra and let $I$ be
its augmentation ideal.  The \emph{bar construction} $\baraug{A}$
on $A$ is a coaugmented differential graded (DG) coalgebra defined
as follows: as a coaugmented coalgebra, it is the tensor coalgebra
$T(SI)$ on $SI$:
\[
T (SI) = k \oplus SI \oplus (SI)^{\otimes 2} \oplus
(SI)^{\otimes 3} \oplus \dotsb.
\]
As is standard, we use bars rather than tensors, and we also conceal
the suspension $s$, writing $[a_{1} | \dotsb | a_{m}]$ for the element
$sa_{1} \otimes \dotsb \otimes sa_{m}$, where $a_{i} \in I$ for each
$i$.  The degree of this element is
\[
\deg [a_{1} | \dotsb | a_{m}] = \left(\sum (-1+\deg_{1} a_{i}),
\sum \deg_{2} a_{i}\right).
\]

The differential $b$ on $\baraug{A}$ is the degree $(1,0)$ map given
as follows: its component $b_{m}: (SI)^{\otimes m} \rightarrow T
(SI)$ is given by
\begin{equation}
\label{1.9} b_m([a_1|\cdots | a_m])=\sum_{j,n} (-1)^{w_{j,n}}
[a_1|\cdots| a_j| \overline{m}_n(a_{j+1}, \cdots, a_{j+n})|
a_{j+n+1}|\cdots| a_m],
\end{equation}
where $\overline{m}_n=(-1)^n m_n$ and
\[
w_{j,n} = \sum_{1\leq s\leq j}(-1+\deg_{1} a_s)+\sum_{1\leq t < n}
(n-t)(-1+\deg_{1} a_{j+t}).
\]
That is, its component mapping $(SI)^{\otimes m}$ to
$(SI)^{\otimes u}$ is
\[
\sum 1^{\otimes j} \otimes (s \circ m_{n} \circ (s^{-1})^{\otimes n})
\otimes 1^{\otimes m-j-n},
\]
where the sum is over pairs $(j,n)$ with $m \geq j+n$, and where
$u=m-n+1$.

If $A$ is an augmented DG algebra, then the above bar construction
is the original bar construction and it is also denoted by $BA$.
\end{definition}

Note that, with this definition, the bar construction of a bigraded
algebra is again bigraded.

\begin{remark}
\label{xxrem1.10}
In \cite[3.6]{Ke01}, Keller describes the bar construction in the
non-augmented situation.  Aside from grading issues, the relation
between his version and ours is as follows: if we write $\barinfty{}$ for
Keller's version, then $\baraug{}$ is the composite
\[
\Aug \rightarrow \ainfty \xrightarrow{\barinfty{}} \DGC \rightarrow
\DGC_{\coaug},
\]
where the first arrow is the equivalence from
Proposition~\ref{xxprop1.7}, and the last arrow takes a coalgebra $C$
to $k\oplus C$, with the apparent coaugmentation.
\end{remark}

The coderivation $b$ encodes all of the higher multiplications of $A$
into a single operation.  Keller
\cite[3.6]{Ke01} notes that if $A$ and $A'$
are augmented $A_\infty$-algebras, then there is a bijection between
Hom sets
\begin{equation}\label{xxeq1.11}
\Aug(A,A') \longleftrightarrow \DGC_{\coaug}
(\baraug{A}, \baraug{A'}).
\end{equation}
(Again, he is working with non-augmented $A_\infty$-algebras,
but Proposition~\ref{xxprop1.7} allows us to translate his
result to this setting.)

We briefly mention the cobar construction.  In full generality, this
would probably take a coaugmented $A_\infty$-coalgebra as input, and
produce an augmented DG algebra.  We have no
interest in working with $A_\infty$-coalgebras, though, and we do
not need this generality.

\begin{definition}
\label{xxdefn1.12}
Given a coaugmented DG coalgebra $C$ with coproduct $\Delta$ and
differential $b_{C}$, the \emph{cobar construction} $\Omega C$ on
$C$ is the augmented DG algebra which as an augmented algebra is the
tensor algebra $T(S^{-1}J)$ on the desuspension of the
coaugmentation coideal $J= \cok (k \rightarrow C)$.
It is graded by putting
\[
\deg [x_{1} | \dotsb | x_{m}] = \left(\sum (1+\deg_{1} x_{i}), \sum
\deg_{2} x_{i}\right).
\]
Its differential is the sum $d=d_{0} + d_{1}$ of the differentials
\[
d_{0} [ x_{1} | \dotsb | x_{m}] = - \sum_{i=1}^{m} (-1)^{n_{i}} [x_{1}
| \dotsb | b_{C}(x_{i}) | \dotsb | x_{m} ],
\]
and
\[
d_{1} [ x_{1} | \dotsb |  x_{m}] =
\sum_{i=1}^{m} \sum_{j=1}^{k_{i}} (-1)^{n_{i} + \deg_{1} a_{ij}}
[ x_{1} | \dotsb | x_{i-1} |  a_{ij} | b_{ij} | \dotsb | x_{m}]
\]
where $n_{i} = \sum_{j<i} (1+\deg_{1} x_{j})$ and
$\sum_{j=1}^{k_{i}} a_{ij} \otimes b_{ij}=
\overline{\Delta}(x_i)$. Here $\overline{\Delta}$ is the induced
coproduct on $J$.
\end{definition}

\begin{definition}
\cite[Section 2.3.4]{Le}
\label{xxdefn1.13}
If $A$ is an augmented $A_\infty$-algebra, then its \emph{enveloping
algebra} $\env{A}$ is defined to be the DG algebra $\env{A} :=
\Omega (\baraug{A})$.
\end{definition}

Thus the enveloping algebra of an augmented $A_\infty$-algebra is an
augmented DG algebra.

\begin{proposition}
\cite[1.3.3.6 and 2.3.4.3]{Le}
\label{xxprop1.14}
There is a natural quasi-isomorphism of $A_\infty$-algebras $A
\rightarrow \env{A}$.
\end{proposition}

The map $A \rightarrow \env{A}$ arises as follows: between the
categories of DG coalgebras and DG algebras, the bar $B$ and cobar
$\Omega$ constructions are adjoint, with $\Omega$ the left adjoint,
and thus for any DG coalgebra $C$, there is a map $C \rightarrow B
(\Omega C)$.  In the case where $C=\baraug{A}$, we get a map
\[
\baraug{A} \rightarrow B (\Omega (\baraug{A})) =
\baraug{(\Omega (\baraug{A}))}.
\]
(One can view an augmented DG algebra $R$ as an $A_\infty$-algebra
with all higher multiplications equal to zero.  In this situation, the
$A_\infty$-bar construction $\baraug{R}$ reduces to the standard DG
algebra bar construction $B(R)$.)  The bijection~\eqref{xxeq1.11} says
that this corresponds to a map
\[
A \rightarrow \Omega (\baraug{A}).
\]
This is the map in Proposition~\ref{xxprop1.14}.  This proposition
says that every augmented $A_\infty$-algebra is quasi-isomorphic to
an augmented DG algebra.  A similar result is also true in the non-augmented
case, although we will not need this.  The quasi-isomorphism between
$A$ and $\Omega (B A)$ is also a standard result in the case when $A$
itself is an augmented DG algebra, although the natural map goes the
other way in that setting; indeed, there is a chain homotopy
equivalence $\Omega B(A) \rightarrow A$ which is a map of DG algebras,
but its inverse need not be an algebra map. See \cite[Section 19]{FHT01},
for example.  One application of Proposition~\ref{xxprop1.14} is that
in the DG case, there is a quasi-inverse in the category $\Aug$.

Also from \cite[Section 19]{FHT01}, we have the following result.

\begin{lemma}
\cite[Section 19]{FHT01}
\label{xxlem1.15}
Let $R$ be an augmented DG algebra.  Assume that $R$ is locally
finite.  Then there is a natural isomorphism $\Omega (R^{\dual}) \cong
(\baraug{R})^{\dual} = B^{\dual}R$.
\end{lemma}

In light of the lemma and Remark~\ref{xxrem1.10}, we point out that
if $R$ is a locally finite augmented associative algebra, then the
homology of the dual of its bar construction is isomorphic to
$\Ext_{R}^{*}(k,k):=\bigoplus_{i\in {\mathbb Z}}\Ext_{R}^{i}(k,k)$.
This is also true when $R$ is an $A_\infty$-algebra;
see \cite[Lemma 11.1]{LP04} and its proof.  Thus by
Theorem~\ref{xxthm1.5}, there is a
quasi-isomorphism of $A_\infty$-algebras $\Ext_{R}^{*}(k,k)
\rightarrow B^{\dual} R$. The $A_\infty$-structure on
$\Ext_{R}^{*}(k,k)$ is studied in \cite{LP08}.

\subsection{Homotopy}
\label{xxsec1.5}

Earlier, we said that we work in the category of augmented
$A_\infty$-algebras.  We also need the homotopy category of such
algebras, and so we need to discuss the notion of homotopy between
$A_\infty$-algebra morphisms.  See
\cite[3.7]{Ke01} and
\cite[1.2.1.7]{Le} for the following.

Let $A$ and $A'$ be augmented $A_\infty$-algebras, and suppose that
$f, g: A \rightarrow A'$ are morphisms of augmented
$A_\infty$-algebras.  Let $F$ and $G$ denote the corresponding maps
$\baraug{A} \rightarrow \baraug{A'}$.  Write $b$ and $b'$ for the
differentials on $\baraug{A}$ and $\baraug{A'}$, respectively.  Then $f$ and
$g$ are \emph{homotopic}, written $f \simeq g$, if there is a map $H:
\baraug{A} \rightarrow \baraug{A'}$ of degree $-1$ such that
\[
\Delta H = (F \otimes H + H \otimes G) \Delta \quad \text{and} \quad
F-G = b' \circ H + H \circ b.
\]
One can also express this in terms of a sequence of maps $h_{n} :
A^{\otimes n} \rightarrow A'$ satisfying some identities, but we will
not need this formulation.  See \cite[1.2.1.7]{Le} for
details (but note that he uses different sign conventions).

Two $A_\infty$-algebras $A$ and $A'$ are \emph{homotopy equivalent}
if there are morphisms $f: A \rightarrow A'$ and $g: A' \rightarrow A$
such that $f \circ g \simeq \id_{A'}$ and $g \circ f \simeq \id_A$.

We will use the following theorem.

\begin{theorem}
\cite{Ka87}, \cite{Pr}, \cite[1.3.1.3]{Le}
\label{xxthm1.16}
\begin{enumerate}
\item Homotopy is an equivalence relation on the set of morphisms of
$A_\infty$-algebras $A \rightarrow A'$.
\item An $A_\infty$-algebra morphism is a quasi-isomorphism
if and only if it is a homotopy equivalence.
\end{enumerate}
\end{theorem}

By part (a), we can define the homotopy category $\Hocat$ to be the
category of augmented $A_\infty$-algebras in which the morphisms are
homotopy classes of maps: that is,
\[
\Hocat (A,A') := \left(\Aug(A,A')/\simeq\right).
\]
By part (b), in this homotopy category, quasi-isomorphisms are
isomorphisms.

\section{The Koszul dual of an $A_\infty$-algebra}
\label{xxsec2}

Let $A$ be an augmented $A_\infty$-algebra.  Its
\emph{$A_\infty$-Koszul dual}, or \emph{Koszul dual} for short, is
defined to be $\koszul{A} := (\baraug{A})^{\dual}$.
By \cite[Section 11 and Lemma 11.1]{LP04}, $\koszul{A}$
is a DG algebra model of the $A_\infty$-$\Ext$-algebra
$\bigoplus_{i\in {\mathbb Z}}\Ext^i_A(k_A,k_A)$ where $k_A$ is
the trivial right $A_\infty$-module over $A$, and where,
by definition, $\Ext^i_A(M,N)=\bigoplus_{j\in {\mathbb Z}}
\InfD(A)(M, S^i \Sigma^j(N))$ for any
right $A_\infty$-modules $M$ and $N$.

In this section, we study some of the basic properties of
$A_\infty$-Koszul duality, we connect it to ``classical'' Koszul
duality, and we discuss a few simple examples.  The main result is
Theorem \ref{xxthmA}, restated as Theorem \ref{xxthm2.4}.

\subsection{Finiteness and connectedness conditions}
\label{xxsec2.1}

In this subsection, we introduce some technical conditions related to
finite-dimensionality and connectivity of bigraded objects.

\begin{definition}
\label{xxdefn2.1}
Let $A$ be an augmented $A_\infty$-algebra and let $I$ be its
augmentation ideal.  We write $I^i_*$ for the direct sum
$I^i_*=\bigoplus_{j} I^i_j$, and similarly $I^*_j =
\bigoplus_{i} I^i_j$.  We say that $A$ is \emph{locally finite} if
each bihomogeneous piece $A^i_j$ of $A$ is finite-dimensional.
We say that $A$ is \emph{strongly locally finite}
if $I$ satisfies the following:
\begin{enumerate}
\item [(1)] each bihomogeneous piece $I^i_j$ of $I$ is
finite-dimensional (i.e., $A$ is locally finite);
\item [(2)] either for all $j \leq 0$, $I^*_j=0$; or for all $j \geq 0$,
$I^*_j=0$; and
\item [(3)] either for all $j$, there exists an $m=m(j)$ so that for all $i
> m(j)$, $I^i_j=0$; or for all $j$, there exists an $m'=m'(j)$ so
that for all $i < m'(j)$, $I^i_j=0$.
\end{enumerate}
We say that $A$ is \emph{Adams connected} if, with respect to the
Adams grading, $A$ is (either positively or negatively) connected
graded and locally finite.  That is,
\begin{itemize}
\item $I^*_j$ is finite-dimensional for all $j$; and
\item either for all $j \leq 0$, $I^*_j=0$, or for all $j \geq 0$,
$I^*_j = 0$.
\end{itemize}
We say that a DG algebra $A$ is \emph{weakly Adams connected} if
\begin{itemize}
\item
the DG bar construction $B(A;A) \cong B(A) \otimes A$ is locally
finite,
\item
the only simple DG $A$-modules are $k$ and its shifts, and
\item
$A$ is an inverse limit of a family of
finite-dimensional left DG $A$-bimodules.
\end{itemize}
\end{definition}

\begin{lemma}
\label{xxlem2.2}
We have the following implications.
\begin{enumerate}
\item Adams connected $\Rightarrow$ strongly locally finite
$\Rightarrow$ weakly Adams connected.
\item $A$ weakly Adams connected $\Rightarrow$ $\koszul{A}$ locally
finite.
\item $A$ strongly locally finite $\Rightarrow$ $\koszul{A}$ strongly
locally finite, and hence every iterated Koszul dual of $A$ is
strongly locally finite.
\item $A$ Adams connected $\Rightarrow$ $\koszul{A}$ Adams connected,
and hence every iterated Koszul dual of $A$ is Adams connected.
\end{enumerate}
\end{lemma}

\begin{proof}
The first implication in part (a) is clear.

For the second implication, if $A$ is strongly locally finite, then
for connectivity reasons, $k$ and its shifts will be the only simple
modules.  We defer the proof that $B(A;A)$ is locally finite until
after the proof of part (c).

For the inverse limit condition, we assume that $I^*_j=0$ when $j \leq
0$, and that for each $j$, there is an $m'(j)$ such that $I^i_j=0$
when $i < m'(j)$.  The other cases are similar.  We will describe a
sequence of two-sided ideals $J_{n}$ in $A$, $n \geq 1$, so that
$A/J_{n}$ is finite-dimensional, and $A = \varprojlim A/J_{n}$.
Define $J_{n}$ to be
\begin{align*}
J_{n} = I^1_{\geq n} &\oplus I^2_{\geq n+m'(1)} \oplus I^3_{\geq n+\min
(2m'(1), m'(2))} \oplus \cdots \\
&\oplus I^s_{\geq n + \min_{\sigma \vdash n-1}(\sum_{a \in
\sigma}m'(a))} \oplus \cdots \oplus I^n_{\geq
n+\min(\cdots)} \\
& \oplus I^{n+1}_{*} \oplus I^{n+2}_{*} \oplus \cdots.
\end{align*}
The notation ``$\sigma \vdash n-1$'' means that $\sigma$ partitions
$n-1$.  The idea here is that, for example, if $J$ contains all of the
elements in bidegrees $(1,\geq n)$, and since $A$ has elements in
bidegrees $(1, \geq m'(1))$, then for $J$ to be an ideal, it should
contain all of the elements in bidegrees $(2,\geq n+m'(1))$.

Part (b) is clear: as graded vector spaces, $B(A;A)$ and $B(A)\otimes_k
A$ are isomorphic, so if $B(A;A)$ is locally finite, so is $B(A)$.

(c) Since $\koszul{A}$ is dual to the tensor coalgebra $T(SI)$, we
focus on $T(SI)$.  If we suppose that $I^*_j=0$ when $j \leq 0$, then
for all $j<n$, we have $(I^{\otimes n})^*_j = 0$.  Shifting $I$ by $S$
does not change this: $((SI)^{\otimes n})^*_j$ will be zero if $j<n$.
Therefore, $T(SI)$ satisfies condition (2) of Definition
\ref{xxdefn2.1}.  Dualizing, we see that $\koszul{A}$ satisfies the
other version of condition (2): if $J$ is its augmentation ideal, then
$J^*_{j}=0$ when $j \geq 0$.
Similarly, if $I^*_j=0$ when $j \geq 0$, then $J^*_{j}=0$ when $j \leq
0$.

So if $I$ satisfies condition (2), then so does $J$.

Now suppose that $I$ satisfies (1) and this version of condition (3):
for each $j$, there is an $m'(j)$ such that $I^i_j=0$ when $i <
m'(j)$.  Then for fixed $n$ and $j$,
\begin{itemize}
\item $(I^{\otimes n})^i_j$ is zero if $i$ is small enough, and
\item $(I^{\otimes n})^i_j$ is finite-dimensional for all $i$.
\end{itemize}
Therefore $T(SI)$ satisfies condition (3).  Furthermore, since for
fixed $j$, $(I^{\otimes n})^*_j$ is zero for all but finitely many
values of $n$, we see that $T(SI)$ satisfies condition (1).
Dualizing, we see that the augmentation ideal $J$ of $\koszul{A}$
satisfies (1) and (3), also (although $J$ satisfies the ``other
version'' of (3)).  This completes the proof of part (c).

We return to part (a): if $A$ is strongly locally finite, then by part
(c), so is the bar construction $B(A)$; more precisely, $A$ and $B(A)$
will satisfy the same version of condition (3).  Hence it is easy to
verify that their tensor product will be locally finite.  This
completes the proof of (a).

(d) $A$ being Adams connected is equivalent to $I$ satisfying (1),
(2), and both versions of (3): for each $j$, there are numbers $m(j)$
and $m'(j)$ so that $I^i_j=0$ unless $m'(j) \leq i \leq m(j)$.  By the
proof of (c), this implies that $\koszul{A}$ satisfies the same
conditions.
\end{proof}

\begin{remark}
\label{xxrem2.3}
One may interchange the roles of $i$ and $j$ in the definition of
strong local finiteness, but the presence of the shift $S$ in
$\koszul{A}=T(SI)$ makes the situation asymmetric.  Suppose
that $I$ satisfies the following:
\begin{enumerate}
\item [(1')] each bihomogeneous piece $I^i_j$ of $I$ of $A$ is
finite-dimensional;
\item [(2')] either for all $i$ there exists an $m=m(i)$ so that for all $j
\geq m(i)$, $I^i_j=0$; or for all $i$ there exists an $m'=m'(i)$ so
that for all $j \leq m'$, $I^i_j=0$; and
\item [(3')] either for all $i \geq 1$, $I^i_*=0$; or for all $i \leq 1$,
$I^i_*=0$.
\end{enumerate}
Then by imitating the proof of part (c) of the lemma, one can show
that $\koszul{A}$ is locally finite; however, it may not satisfy
(2').

Suppose that $I$ satisfies (1'), (2'), and the following:
\begin{enumerate}
\item [(3'')] either for all $i \geq 0$, $I^i_*=0$; or for all $i \leq 1$,
$I^i_*=0$.
\end{enumerate}
Then the same proof shows that the augmentation ideal of $\koszul{A}$
satisfies (1'), (2'), and (3'') as well, and hence the same holds for
every iterated Koszul dual of $A$.
\end{remark}

\subsection{$A_\infty$-Koszul duality}
\label{xxsec2.2}

Here is the main theorem of this section, which is a slight
generalization of \cite[Theorem 11.2]{LP04}.  This is
Theorem~\ref{xxthmA} from the introduction.

\begin{theorem}
\label{xxthm2.4}
Suppose that $A$ is an augmented $A_\infty$-algebra with
$\koszul{A}$ locally finite.  Then there is a natural
quasi-isomorphism of $A_\infty$-algebras $A \cong
\koszul{\koszul{A}}$.
\end{theorem}

If $A$ is weakly Adams connected, then essentially by definition (or
see Lemma \ref{xxlem2.2}), $\koszul{A}$ is locally finite.  Hence
by Lemma \ref{xxlem2.2}, $\koszul{A}$ is locally finite if $A$ is
Adams connected or strongly locally finite.  The summary of the
theorem's proof is that the double Koszul dual is the enveloping
algebra $\env{A}$ of $A$ (see Definition~\ref{xxdefn1.13} and
Proposition~\ref{xxprop1.14}).

\begin{proof} By definition, the double Koszul dual
$\koszul{\koszul{A}}$ is
$(\baraug{((\baraug{A})^{\dual})})^{\dual}$.  Apply
Lemma~\ref{xxlem1.15} to the DG algebra
$\koszul{A}=(\baraug{A})^{\dual}$,
which is locally finite.  Then there are natural
DG algebra isomorphisms
\[
\koszul{\koszul{A}}= (\baraug{((\baraug{A})^{\dual})})^{\dual}
\cong \Omega (((\baraug{A})^{\dual})^{\dual}) \cong
\Omega (\baraug{A}).
\]
Proposition~\ref{xxprop1.14} gives a natural
$A_\infty$-isomorphism $A \xrightarrow{\cong} \Omega
(\baraug{A})$.
\end{proof}

Koszul duality $\koszul{-}$ is a contravariant functor from
$A_\infty$-algebras to DG algebras, and since one can view a
DG algebra as being an $A_\infty$-algebra, there are functions
\[
\Aug (A,A') \xrightarrow{E(-)}
\DGA_{\aug} (\koszul{A'},\koszul{A}) \xrightarrow{}
\Aug (\koszul{A'}, \koszul{A}),
\]
for augmented $A_\infty$-algebras $A$ and $A'$.

\begin{corollary}
\label{xxcor2.5}
Suppose that $A$ and $A'$ are augmented $A_\infty$-algebras with
$\koszul{A}$, $\koszul{\koszul{A}}$, $\koszul{A'}$, and
$\koszul{\koszul{A'}}$ locally finite.
\begin{enumerate}
\item Then $\koszul{-}$ gives a bijection
\[
\Aug (A,A') \xrightarrow{\sim} \DGA_{\aug}(\koszul{A'}, \koszul{A}).
\]
\item The composite
\[
\Aug (A,A') \xrightarrow{\sim} \DGA_{\aug}(\koszul{A'}, \koszul{A})
\rightarrow \Aug (\koszul{A'}, \koszul{A})
\]
induces a bijection
\[
\Hocat (A,A') \xrightarrow{\sim} \Hocat (\koszul{A'}, \koszul{A}).
\]
\item Hence every $A_\infty$-algebra map $f: \koszul{A'} \rightarrow
\koszul{A}$ is homotopic to a DG algebra map.
\end{enumerate}
\end{corollary}

\begin{proof}
(a) From \eqref{xxeq1.11} we have a bijection
\[
\Aug(A,A') \xrightarrow{\sim} \DGC_{\coaug} (\baraug{A},
\baraug{A'}).
\]
We are assuming that $\baraug{A}$ and $\baraug{A'}$ are locally
finite, so the vector space duality maps
\begin{align*}
\DGC_{\coaug} (\baraug{A}, \baraug{A'}) &\rightarrow \DGA_{\aug}
((\baraug{A'})^{\dual}, (\baraug{A})^{\dual}) \\
 & \rightarrow
\DGC_{\coaug} (((\baraug{A})^{\dual})^{\dual}, ((\baraug{A'})^{\dual})^{\dual})
\end{align*}
are bijections.  Therefore so is
\[
\Aug(A,A') \xrightarrow{\sim} \DGA_{\aug} (\koszul{A'}, \koszul{A}).
\]

(b,c) The naturality of the quasi-isomorphism (= homotopy equivalence) in
Theorem \ref{xxthm2.4} says that the function
\begin{align*}
\Hocat (A, A') & \xrightarrow{\koszul{\koszul{-}}} \Hocat
(\koszul{\koszul{A}}, \koszul{\koszul{A'}}) \\
f &\longmapsto \koszul{\koszul{f}}
\end{align*}
is a bijection.  That is, the composite
\[
\Hocat (A, A') \rightarrow \Hocat (\koszul{A'}, \koszul{A})
\rightarrow \Hocat (\koszul{\koszul{A}}, \koszul{\koszul{A'}})
\]
is a bijection.  The first map here is induced by
\[
\Aug (A,A') \xrightarrow{\sim} \DGA_{\aug}(\koszul{A'}, \koszul{A})
\xrightarrow{i} \Aug (\koszul{A'}, \koszul{A}),
\]
and the second by
\[
\begin{split}
\Aug (\koszul{A'}, \koszul{A}) & \xrightarrow{\sim} \DGA_{\aug}
(\koszul{\koszul{A}}, \koszul{\koszul{A'}}) \\
 & \xrightarrow{j} \Aug (\koszul{\koszul{A}}, \koszul{\koszul{A'}}).
\end{split}
\]
Since both of the functions $i$ and $j$ are inclusions, the functions
$\Ho \,i$ and $\Ho \,j$ must be bijections.  This proves (b) and (c).
\end{proof}

\subsection{Classical Koszul duality}
\label{xxsec2.3}

Classically, a \emph{Koszul algebra} is a connected graded associative
algebra $R$ which is locally finite, is generated in degree 1, has
quadratic relations, and has its $i$th graded Ext-group
$\Ext_{R}^{i}(k,k)$ concentrated in degree $-i$ for each $i$; see
\cite[Theorem 2.10.1]{BGSo} or \cite[Theorem 5.9(6)]{Sm}, for example.
(In those papers, $\Ext_{R}^{i}(k,k)$ is actually required to be
concentrated in degree $i$, but that is the result of different
grading conventions.)  Its (classical) Koszul dual, also denoted
by $R^{!}$, is $\Ext_{R}^{*}(k,k)$.
One can show that if $R$ is a Koszul
algebra, then so is $R^{!}$ -- see \cite[2.9.1]{BGSo}, for example.

A standard example is an exterior algebra $R=\Lambda (x_{1}, \dots,
x_{n})$ on generators $x_{i}$ each in degree $1$; then its Koszul dual
$R^{!}$ is the polynomial algebra $k[y_{1}, \dots, y_{n}]$, with each
$y_{i}$ in degree $(1,-1)$.

We want all of our algebras to be bigraded, though, and we want the
double Koszul dual to be isomorphic, as a bigraded algebra, to the
original algebra.  Thus we might grade $\Lambda (x_{1}, \dots, x_{n})$
by putting each $x_{i}$ in degree $(0,1)$, in which case
$R^{!}=k[y_{1}, \dots, y_{n}]$.  The grading for $R^{!}$ is given as
follows: $y_{i}$ is represented in the dual of the bar construction
for $R$ by the dual of $[x_{i}]$, and since $\deg [x_{i}] = (-1,1)$,
$y_{i}$ has degree $(1,-1)$.  The double Koszul dual is exterior on
classes dual to $[y_{i}]$ in the bar construction on $R^{!}$, each of
which therefore has degree $(0,1)$.

Note that these are graded in such a way that there are no possible
nonzero higher multiplications $m_{n}$ on them.  This absence of
higher multiplications is typical for a Koszul algebra, as Keller
\cite[3.3]{Ke01} and the authors
\cite[Section 11]{LP04} point out.  Conversely,
if we grade our algebras in such a way that there are no possible
higher multiplications, we can recover classical Koszul duality.

\begin{definition}
\label{xxdefn2.6}
Fix a pair of integers $(a,b)$ with $b \neq 0$.  A bigraded
associative algebra $A$ is an \emph{$(a,b)$-generated Koszul algebra} if
it satisfies these conditions:
\begin{enumerate}
\item $A_{0,0} = k$,
\item $A$ is locally finite,
\item $A$ is generated in bidegree $(a,b)$,
\item the relations in $A$ are generated in bidegree $(2a,2b)$,
\item for each $i$, the graded vector space $\Ext_A^{i}(k,k)$ is
concentrated in degree $(i(a-1), -ib)$.
\end{enumerate}
\end{definition}

In fact, conditions (c) and (d) should follow from condition (e): one
should be able to imitate the proofs of \cite[2.3.1 and 2.3.2]{BGSo}.

If $A$ is a bigraded associative algebra, the classical Koszul
dual of $A$, denoted by $A^{!}$, is defined to be
$H\koszul{A}$ -- the \emph{homology} of the $A_\infty$-Koszul dual
$\koszul{A}=(\baraug{A})^{\dual}$. Forgetting grading issues,
$A^{!}$ is isomorphic to $\Ext_A^{*}(k,k)$. In classical
ring theory, we often consider the classical Koszul dual as
an associative algebra -- an $A_\infty$-algebra with $m_{n}=0$ if
$n\neq 2$.

\begin{corollary}
\label{xxcor2.7}
Fix a pair of integers $(a,b)$ with $b \neq 0$.  If $A$ is an
$(a,b)$-generated Koszul algebra, then $\koszul{A}$ and
$\koszul{\koszul{A}}$
are quasi-isomorphic to associative algebras $A^{!}$
and  $(A^{!})^{!}$, respectively, and there is an
isomorphism of bigraded algebras $A \cong (A^{!})^{!}$.
\end{corollary}

This result is known \cite{BGSo} so we only give a sketch of proof.

\begin{proof}[Sketch of proof of Corollary \ref{xxcor2.7}]
By Theorem~\ref{xxthm1.5}, the $A_\infty$-Koszul dual $\koszul{A}$
is quasi-isomorphic to $A^{!}$ with some possible higher
multiplications. We need to show, among other things, that in
this case, the higher multiplications on $A^{!}$ are zero.

Since $b \neq 0$, both $\koszul{A}$ and $\koszul{\koszul{A}}$ will be
locally finite.

For any nonzero non-unit element $x \in \Ext_A^{*}(k,k)$,
its bidegree $(\deg_{1} x,
\deg_{2} x)$ satisfies
\begin{equation*}
\tag{$*$}
\frac{\deg_{1} x}{\deg_{2} x} = - \frac{a-1}{b},
\end{equation*}
and this fraction makes sense since $b \neq 0$.  The same is true for
any tensor product of such elements.  Since the higher multiplication
$m_{n}$ has degree $(2-n,0)$, one can see that if $n\neq 2$, the
bidegree of $m_{n}(x_{1} \otimes \dotsb \otimes x_{n})$ will not
satisfy $(*)$, and so will be zero.  Thus there is no
nonzero higher multiplications on $A^{!}$ which
is compatible with the bigrading. This implies that
the $A_\infty$-algebra $\koszul{A}$ is quasi-isomorphic to
the associative algebra $A^{!}$.

Now we claim that $A^{!}$ is $(1-a,-b)$-generated Koszul.  There is an obvious
equivalence of categories between $\Z$-graded algebras and $\Z \times
\Z$-graded algebras concentrated in degrees $(na,nb)$ for $n \in \Z$;
under that equivalence, $(a,b)$-generated Koszul algebras correspond to Koszul
algebras in the sense of \cite{BGSo}.  Koszul
duality takes $\Z$-graded algebras generated in degree 1 to $\Z
\times \Z$-graded algebras generated in degree $(1,-1)$.  It takes
$\Z \times \Z$-graded algebras generated in degree $(a,b)$ to $\Z
\times \Z$-graded algebras generated in degree $(1-a,-b)$.  The proof
of \cite[2.9.1]{BGSo} carries over to show that
since $A$ is $(a,b)$-generated Koszul, its dual $A^{!}$ is
$(1-a,-b)$-generated Koszul.

Since $A^{!}$ is Koszul, its Koszul dual $(A^{!})^{!}$
is associative (or there is no nonzero higher multiplications on
$(A^{!})^{!}$, by the first part of the proof).  A similar
grading argument shows that any morphism $f: A \rightarrow
(A^{!})^{!}$ of $A_\infty$-algebras must be strict; thus the
isomorphism of $A_\infty$-algebras $A \rightarrow
(A^{!})^{!}$ is just an isomorphism of associative algebras.
\end{proof}

\subsection{Examples: exterior and polynomial algebras}
\label{xxsec2.4}

In this subsection, we consider some simple examples involving exterior
algebras and polynomial algebras.  The first example shows that in the
classical setting, it is crucial that a Koszul algebra be generated in
a single degree.

\begin{example}
\label{xxex2.8}
Assume that the ground field $k$ has characteristic 2, and
consider the exterior algebra $\Lambda = \Lambda (x_{1}, x_{2})$ with
$\deg x_{i} = (0,i)$: this is \emph{not} a classical Koszul algebra,
nor is it an $(a,b)$-generated Koszul algebra, since there are generators in
multiple degrees.  The same goes for $\Lambda \otimes \Lambda$.  The
Ext-algebra for $\Lambda$ is the polynomial algebra
$\Lambda^{!}=k[y_{1}, y_{2}]$ with $\deg y_{i} = (1,-i)$.  Similarly,
the Ext-algebra for $\Lambda \otimes \Lambda$ is isomorphic to
$\Lambda^{!} \otimes \Lambda^{!}$.  Although it is true that
$(\Lambda^{!})^{!} \cong \Lambda$, there are naturality problems.  In
particular, the map
\[
\Alg (\Lambda, \Lambda \otimes \Lambda) \rightarrow \Alg (\Lambda^{!}
\otimes \Lambda^{!}, \Lambda^{!})
\]
is not injective: one can show that the following two maps induce the
same map on Ext:
\begin{align*}
f: x_{1} & \mapsto x_{1} \otimes 1 + 1 \otimes x_{1}, &
g: x_{1} & \mapsto x_{1} \otimes 1 + 1 \otimes x_{1}, \\
f: x_{2} & \mapsto x_{2} \otimes 1 + 1 \otimes x_{2}, &
g: x_{2} & \mapsto x_{2} \otimes 1 + 1 \otimes x_{2} + x_{1} \otimes x_{1}.
\end{align*}
(Indeed, any algebra map $\Lambda \rightarrow \Lambda \otimes \Lambda$
gives a coproduct on $\Lambda$, and any coproduct on $\Lambda$ induces
the Yoneda product on $\Ext_{\Lambda}^{*}(k,k)$.)
This shows the importance of the requirement that Koszul algebras
be generated in a single degree.

Now, Theorem~\ref{xxthm2.4} and Corollary~\ref{xxcor2.5}
apply here.  The $A_\infty$-version of the Koszul dual of $\Lambda$
is quasi-isomorphic to $\Lambda^{!}$: $\koszul{\Lambda} \cong
\Lambda^{!} = k[y_{1}, y_{2}]$,
where the $A_\infty$-structure on $\Lambda^{!}$ is given by
$m_{n}=0$ when $n \neq 2$.  Corollary~\ref{xxcor2.5} says that
there is a bijection
\[
\Hocat (\Lambda, \Lambda \otimes \Lambda)
\xrightarrow{\sim} \Hocat (\Lambda^{!} \otimes \Lambda^{!},
\Lambda^{!}).
\]
This fixes the flaw above; the two (strict) maps $f$ and $g$
correspond to $A_\infty$-algebra morphisms $\koszul{f}$ and
$\koszul{g}$, and while $\koszul{f}_{1} = \koszul{g}_{1}$, the
morphisms must differ in some higher component.  (Even though the
algebras involved here have $A_\infty$-structures with zero higher
multiplications, there are non-strict $A_\infty$-algebra morphisms
between them.  Also, $k[y_{1},y_{2}]$ is quasi-isomorphic, not equal,
to the $A_\infty$-Koszul dual of $\Lambda$, so
Corollary~\ref{xxcor2.5}(c), which says that every non-strict map on
Koszul duals is homotopic to a strict one, does not apply here.)
\end{example}

\begin{example}
\label{xxex2.9}
Let $A=\Lambda (x)$ with $\deg x = (a,0)$; Corollary~\ref{xxcor2.7}
does not apply in this case.  The Koszul dual is $\koszul{A}=k[y]$
with $\deg y = (1-a,0)$ and with $m_n=0$ for $n \neq 2$.  Assume that
$a \neq 0,1$; then $\koszul{A}$ is locally finite, as is
$\koszul{\koszul{A}}$ by Remark~\ref{xxrem2.3}.  Thus
Corollary~\ref{xxcor2.5} says that there is a bijection
\[
\Hocat (\Lambda (x), \Lambda (x)) \xrightarrow{\sim}
\Hocat (k[y], k[y]).
\]
For degree reasons, every map $\Lambda (x) \rightarrow \Lambda (x)$
must be strict, so each map is given by the image of $x$: $x
\longmapsto cx$ for any scalar $c \in k$.

On the other hand, degree reasons do not rule out non-strict maps
$k[y] \rightarrow k[y]$.  Strict maps will correspond to those from
$\Lambda (x)$ to itself, with the map given by the scalar $c$
corresponding to the map $y \mapsto cy$.  Thus, as pointed out in
Corollary~\ref{xxcor2.5}(c), if there are any non-strict maps, then they are
homotopic to strict ones.  In particular, one can see that if two
$A_\infty$-algebra maps $f, g: k[y] \rightarrow k[y]$ are homotopic,
then $f_{1}=g_{1}$.  So if $f=(f_{1}, f_{2}, \dots)$ is such a map,
then it will be homotopic to the strict map $(f_{1}, 0, 0, \dots)$.
None of this is immediately clear from the morphism and homotopy
identities, so Koszul duality, in the form of
Corollary~\ref{xxcor2.5}, gives some insight into $A_\infty$-maps
from $k[y]$ to itself.
\end{example}

\begin{example}
\label{xxex2.10}
Now consider $A=\Lambda (x)$ with $\deg x = (1,0)$.  As in the
previous example, every map $\Lambda (x)\rightarrow \Lambda (x)$ must
be strict.  In this case, $\baraug{A}$ is the vector space spanned by
the classes $[x|\dotsb |x]$, all of which are in bidegree $(0,0)$.
Since $\baraug{A}$ is not locally finite, Theorem~\ref{xxthm2.4} does
not apply, and the Koszul dual $\koszul{A}$ ends up being the power
series ring $k[[y]]$ instead of the polynomial ring $k[y]$.  Consider
the composite
\[
\Aug (A,A) \xrightarrow{\sim} \DGC_{\coaug} (\baraug{A},
\baraug{A}) \rightarrow \DGA^{\aug} (\koszul{A}, \koszul{A}).
\]
The first map is a bijection by \eqref{xxeq1.11}, but the second map
is not, essentially since the map is given by vector space duality and
the vector spaces involved are not finite-dimensional.  Since strict
$A_\infty$-maps are homotopic if and only if they are equal, we get a
proper inclusion
\[
\Hocat (A,A) \overset{\neq}{\hookrightarrow} \Hocat (\koszul{A},
\koszul{A}).
\]
Thus Corollary~\ref{xxcor2.5} fails here.
\end{example}

\begin{example}
\label{xxex2.11}
We also mention the case when $A=\Lambda (x)$ with $\deg x = (a,b)$
with $b \neq 0$.  In this case, $\koszul{A} = k[y]$ with $\deg y =
(1-a,-b)$.  Also, $A_\infty$-morphisms $\Lambda (x) \rightarrow \Lambda
(x)$ and $k[y] \rightarrow k[y]$ must be strict, and it is easy to
show that the strict maps are in bijection, as
Corollary~\ref{xxcor2.5} says they should be.  Indeed in this
case, classical Koszul duality (Corollary~\ref{xxcor2.7})
applies, since $\Lambda (x)$ is $(a,b)$-generated Koszul with $b \neq 0$.
\end{example}

\part{Koszul duality for modules}

\section{Background on $A_\infty$-modules}
\label{xxsec3}

Koszul duality relates not just to algebras, but also to modules over
them.  In this section, we briefly review the relevant categories of
modules over an $A_\infty$-algebra.  See Keller \cite[4.2]{Ke01} for
a few more details, keeping in mind that since he is not working in
the augmented setting, a little translation is required, especially in
regards to the bar construction.  The paper \cite[Section 6]{LP04}
also has some relevant information, as does Lef\`evre-Hasegawa's
thesis \cite{Le}.

There are several subsections here: the definition of
$A_\infty$-module; the bar construction; derived categories;
and several sections about ``opposites.''

\subsection{$A_\infty$-modules}

Let $A$ be an $A_\infty$-algebra.  A bigraded vector space $M$ is a
\emph{right $A_\infty$-module over $A$} if there are graded maps
\[
m_{n} : M \otimes A^{\otimes n-1} \rightarrow M, \quad n \geq 1,
\]
of degree $(2-n,0)$ satisfying the Stasheff identities $\mathbf{SI(n)}$,
interpreted appropriately. Similarly, a bigraded vector space $N$ is a
\emph{left $A_\infty$-module over $A$} if there are graded maps
\[
m_{n} : A^{\otimes n-1} \otimes N \rightarrow N, \quad n \geq 1,
\]
of degree $(2-n,0)$ satisfying the Stasheff identities $\mathbf{SI(n)}$,
interpreted appropriately.

Morphisms of right $A_\infty$-modules are defined in a similar way:
a \emph{morphism $f: M \rightarrow M'$ of right $A_\infty$-modules
over $A$} is a sequence of graded maps
\[
f_{n} : M \otimes A^{\otimes n-1} \rightarrow M', \quad n \geq 1,
\]
of degree $(1-n,0)$ satisfying the Stasheff morphism identities
$\mathbf{MI(n)}$.
Morphisms of left $A_\infty$-modules are defined analogously, and so
are homotopies in both the right and left module settings.

Now suppose that $A$ is an augmented $A_\infty$-algebra.  A right
$A_\infty$-module $M$ over $A$ is \emph{strictly unital} if for all
$x \in M$ and for all $a_i\in A$, $m_{2}(x \otimes 1) = x$ and
\[
m_{n}(x \otimes a_{2} \otimes \dotsb \otimes a_{n}) = 0
\]
if $a_{i}=1$ for some $i$.  A morphism $f$ of such is \emph{strictly
unital} if for all $n \geq 2$, we have $f_{n}(x \otimes a_{2} \otimes
\dotsb \otimes a_{n})=0$ if $a_{i}=1$ for some $i$.

Given an augmented $A_\infty$-algebra $A$, let $\Mod^{\infty} (A)$
denote the category of strictly unital right $A_\infty$-modules with
strictly unital morphisms over $A$.

Suppose $A$ is an augmented $A_\infty$-algebra.  The morphism
$\varepsilon : A \rightarrow k$ makes the vector space $k$ into a left
$A_\infty$-module over $A$.  It is called the \emph{trivial left
$A_\infty$-module over $A$} and is denoted by $_Ak$.  The
\emph{trivial right $A_{\infty}$-module over $A$} is defined
similarly, and is denoted by $k_A$.

\subsection{The bar construction for modules}\label{xxsec3.2}

The bar construction is as useful for $A_\infty$-modules as it is
for $A_\infty$-algebras: recall that the bar construction on $A$ is
$\baraug{A} = T(SI)$.  A strictly unital right $A_\infty$-module
structure on a bigraded vector space $M$ gives a comodule differential
on the right $\baraug{A}$-comodule
\[
\baraug{(M;A)} := SM \otimes T(SI),
\]
as in Definition~\ref{xxdefn1.8}.  Also, morphisms of right
modules $M \rightarrow M'$ are in bijection with morphisms of right DG
comodules $\baraug{(M;A)} \rightarrow \baraug{(M';A)}$ as in
\eqref{xxeq1.11}, and the notion of homotopy translates as well.

Similarly, one has a bar construction for left $A_\infty$-modules,
defined by
\[
\baraug{(A;N)} := T(SI) \otimes SN,
\]
with the same formula for the differential.

\subsection{Derived categories}

Let $\InfD(A)$ be the derived category associated to the module
category $\Mod^{\infty}(A)$.  According to
\cite[4.2]{Ke01} (see also
\cite[2.4.2]{Le}), this derived category is the same as
the homotopy category of $\Mod^{\infty}(A)$: in the homotopy category
for $A_\infty$-modules, quasi-isomorphisms have already been
inverted.  This statement is the module version of
Theorem~\ref{xxthm1.16}.

Given a DG algebra $R$, write $\Mod \,R$ for the category of unital
DG right $R$-modules, and write $\D (R)$ for its derived category.
A good reference for the derived category $\D(R)$ is \cite{Ke94};
also see \cite{Ke07, KM}. See \cite[2.4.3 and 4.1.3]{Le} and
\cite[7.2 and 7.3]{LP04} for the following.

\begin{proposition}
\label{xxprop3.1}
\begin{enumerate}
\item Suppose that $A$ and $B$ are augmented $A_\infty$-algebras.
If $f: A \rightarrow B$ is an $A_\infty$-isomorphism, then the
induced functor $f^{*}: \InfD(B) \rightarrow \InfD(A)$ is a
triangulated equivalence.
\item Suppose that $R$ is an augmented DG algebra.  Then the inclusion
$\Mod \,R \hookrightarrow \Mod^{\infty} (R)$ induces a triangulated
equivalence $\D (R) \rightarrow \InfD (R)$.
\item Hence if $A$ is an augmented $A_\infty$-algebra and $R$ is an
augmented DG algebra with an $A_\infty$-isomorphism $A \rightarrow R$,
there is a triangulated equivalence $F: \D (R) \rightarrow \InfD(A)$.
Under this equivalence, $F(k_R) \cong k_A$ and $F(R) \cong A$.
\end{enumerate}
\end{proposition}

Hence in the category $\InfD(A)$ one can perform many of the
usual constructions, by first working in the derived category $\D
(\env{A})$ of its enveloping algebra and then applying the equivalence
of categories $\D (\env{A}) \rightarrow \InfD(A)$.  See
Section~\ref{xxsec5} for an application of this idea.
We note that the Adams shift is an automorphism of
$\InfD(A)$.

\subsection{The opposite of a DG algebra}
\label{xxsec3.4}

If $A=(A,m_{1},m_{2})$ is a DG algebra with differential $m_{1}$ and
multiplication $m_{2}$, we define the \emph{opposite algebra} of $A$
to be $(A^{\op}, m_{1}^{\op}, m_{2}^{\op})$, where $A^{\op}=A$,
$m_{1}^{\op}=-m_{1}$, and
\[
m_{2}^{\op}(a \otimes b) = (-1)^{(\deg a) (\deg b)} m_{2}(b \otimes a).
\]
That is, $m_{2}^{\op}=m_{2} \circ \tau$, where $\tau$ is the twist
function, which interchanges tensor factors at the expense of the
appropriate Koszul sign.  One can verify that this is a DG algebra:
$m_{1}^{\op}$ and $m_{2}^{\op}$ satisfy a Leibniz
formula.  (Choosing $m_{1}^{\op}=m_{1}$ also works, but is not
compatible with the bar construction: see below.)

If $f: A \rightarrow A'$ is a map of DG algebras, define $f^{\op} :
A^{\op} \rightarrow (A')^{\op}$ by $f^{\op}=f$.  Then $f^{\op}$ is
also a DG algebra map, so $\op$ defines an automorphism of the
category of DG algebras, and it is clearly its own inverse.

Dually, given a DG coalgebra $C=(C,d,\Delta)$, we define its
\emph{opposite coalgebra} to be $(C^{\op}, d^{\op}, \Delta^{\op})$,
where $C^{\op}=C$, $d^{\op}=-d$, and $\Delta^{\op}= \tau \circ
\Delta$.  With these definitions, for any DG algebra $A$ there is an
isomorphism
\begin{align*}
B(A^{\op}) & \xrightarrow{\ \Phi\ } B(A)^{\op}, \\
[a_{1} | \dots | a_{m}] &\longmapsto (-1)^{\sum_{i<j}(-1+\deg
a_{i})(-1+\deg a_{j})} [a_{m} | \dots | a_{1}],
\end{align*}
of DG coalgebras.  Note that, had we defined the differential
$m_{1}^{\op}$ in $A^{\op}$ by $m_{1}^{\op}=m_{1}$, this map $\Phi$
would not be compatible with the differentials, and so would not be a
map of DG coalgebras.

If $F: (C,d,\Delta) \rightarrow (C', d', \Delta')$ is a morphism of DG
coalgebras, then define $F^{\op}$ to be equal to $F$.  One can check
that $F^{\op}:C^{\op} \rightarrow (C')^{\op}$ is also a map of DG
coalgebras, making $\op$ into a functor $\op : \DGC \rightarrow \DGC$.
As above, it is an automorphism which is its own inverse.

The definition of homotopy in Subsection~\ref{xxsec1.5} works in
general for DG coalgebra morphisms: if $F$ and $G$ are DG
coalgebra morphisms $C \rightarrow C'$, then a \emph{homotopy from
$F$ to $G$} is a map $H: C \rightarrow C'$ of degree $-1$ such that
\[
\Delta H = (F \otimes H + H \otimes G) \Delta, \quad
F-G= d' \circ H + H \circ d.
\]
We write $H=H(F\rightarrow G)$ to indicate the ``direction'' of the
homotopy.  One can check that, in this situation, the map $H$ also
defines a homotopy from $G^{\op}$ to $F^{\op}$, as maps $C^{\op}
\rightarrow (C')^{\op}$.  Therefore, we may define $H^{\op}(G^{\op}
\rightarrow F^{\op})$ to be $H(F\rightarrow G)$.  As a consequence,
$\op$ induces an automorphism on the homotopy category of DG
coalgebras.

\subsection{The opposite of an $A_\infty$-algebra}

Now we define a functor
\[
\op : \ainfty \rightarrow \ainfty
\]
which generalizes the opposite functor on the category of DG algebras.
Given an $A_\infty$-algebra $(A, m_{1}, m_{2}, m_{3}, \dotsc)$,
define $(A^{\op}, m_{1}^{\op}, m_{2}^{\op}, m_{3}^{\op}, \dotsc)$ as
follows: as a bigraded vector space, $A^{\op}$ is the same as $A$.  The
map $m_{n}^{\op}: (A^{\op})^{\otimes n} \rightarrow A^{\op}$ is
defined by $m_{n}^{\op}=(-1)^{\varepsilon (n)} m_{n} \circ
\text{(twist)}$, where ``twist'' is the map which reverses the factors
in a tensor product, with the appropriate Koszul sign, and
\[
\varepsilon (n) = \begin{cases}
1, & \text{if $n \equiv 0,1 \pmod{4}$,}\\
0, & \text{if $n \equiv 2,3 \pmod{4}$.}
\end{cases}
\]
Equivalently, since only the parity of $\varepsilon (n)$ is
important, $\varepsilon (n)=\binom{n+2}{2}$, or $\varepsilon (n) =
\binom{n}{2}+1$.  Thus when applied to elements,
\[
m_{n}^{\op}(a_{1} \otimes \dotsb \otimes a_{n}) =
(-1)^{\varepsilon (n) + \sum_{i<j} (\deg a_{i}) (\deg a_{j})}
m_{n}(a_{n} \otimes \dotsb \otimes a_{1}).
\]

\begin{lemma}\label{xxlem3.2}
The function $\varepsilon$ satisfies the following additivity
formula: for any $q \geq 1$ and any $i_{s} \geq 1$, $s=1,
2, \dots, q$,
\[
\sum_{1 \leq s \leq q} \varepsilon (i_{s}) + \varepsilon \left(\sum_{1 \leq
s \leq q} i_{s} - q + 1\right) + \sum_{1 \leq s < t \leq q} (i_{s} -
1)(i_{t} - 1) \equiv q+1 \pmod{2}.
\]
\end{lemma}

\begin{proof}
The $q=1$ case is trivial, the $q=2$ case may be established by (for
example) considering the different congruence classes of $i_{1}$ mod
4, and for larger $q$, one can use a simple induction argument.
\end{proof}

\begin{lemma}\label{xxlem3.3}
$(A^{\op}, m_{1}^{\op}, m_{2}^{\op}, m_{3}^{\op}, \dotsc)$
is an $A_\infty$-algebra.
\end{lemma}

\begin{proof}
We need to check that $(A^{\op}, m_{1}^{\op}, m_{2}^{\op},
m_{3}^{\op}, \dotsc)$ satisfies the Stasheff identities.
This is a tedious, but straightforward, verification, which we leave
to the reader.  The $q=2$ case of Lemma~\ref{xxlem3.2} is useful.
\end{proof}

We also need to specify what happens to morphisms.  Given a morphism
$f: A \rightarrow B$, we define $f^{\op}: A^{\op} \rightarrow B^{\op}$
by defining
\[
f^{\op}_{n}: (A^{\op})^{\otimes n} \rightarrow B^{\op}
\]
to be $f^{\op}_{n}= (-1)^{1+\varepsilon (n)} f_{n} \circ
\text{(twist)}$; that is,
\[
f^{\op}_{n} (a_{1} \otimes \dots \otimes a_{n}) =
(-1)^{1+\varepsilon (n) + \sum_{i<j} \deg a_{i} \deg a_{j}}
f_{n} (a_{n} \otimes \dots \otimes a_{1}).
\]

\begin{lemma}\label{xxlem3.4}
The family $f^{\op} = (f_{n}^{\op})$ is a morphism of
$A_\infty$-algebras.
\end{lemma}

\begin{proof}
This is another tedious verification.
Lemma~\ref{xxlem3.2} is used here.
\end{proof}

To complete this circle of ideas, we should consider the bar
construction.  That is, consider the following diagram of functors:
\[
\xymatrix{
\ainfty \ar[d]^{B(-)} \ar[rr]^{\op}_{\sim}
& & \ainfty \ar[d]^{B(-)} \\
\DGC \ar[r]^{\op}_{\sim}
& \DGC
& \DGC
}
\]
The horizontal arrows are equivalences of categories.  The vertical
arrows are fully faithful embeddings.  Starting with an
$A_\infty$-algebra $A$ in the upper left corner, mapping down and
then to the right gives $B(A)^{\op}$, while mapping to the right and
then down gives $B(A^{\op})$.  It would be nice if these two DG
coalgebras agreed, and indeed they do.

\begin{lemma}\label{xxlem3.5}
For any $A_\infty$-algebra $A$, the map
\begin{align*}
\Phi: B(A^{\op}) &\xrightarrow{\quad} B(A)^{\op}, \\
[a_{1} | \dots | a_{m}] & \longmapsto (-1)^{\sum_{i<j}(-1+ \deg
a_{i})(-1 + \deg a_{j})} [a_{m} | \dots | a_{1}]
\end{align*}
is an isomorphism of DG coalgebras.
\end{lemma}

Note that $B(A)^{\op}$ is the opposite coalgebra to $B(A)$, as
defined in Subsection~\ref{xxsec3.4}.

\begin{proof}
Left to the reader.
\end{proof}

This result gives us a second way to prove
Lemma~\ref{xxlem3.4}, that the definition
$f_{n}^{\op}=(-1)^{1+\varepsilon (n)}f_{n} \circ \text{(twist)}$
defines a morphism of $A_\infty$-algebras: one just has to check
that if the $A_\infty$-algebra morphism $f: A \rightarrow A'$
corresponds to the DG coalgebra morphism $B(f): B(A) \rightarrow B(A')$,
then the composite
\[
B(A)^{\op} \xrightarrow[\cong]{\Phi^{-1}} B(A^{\op})
\xrightarrow{B(f^{\op})} B((A')^{\op}) \xrightarrow[\cong]{\Phi}
B(A')^{\op}
\]
is equal to $B(f)^{\op}=B(f)$.  This is straightforward.

Once we know that the bar construction works well with opposites, we
can define the opposite of a homotopy between $A_\infty$-algebra
maps in terms of the bar construction.  Thus $\op$ defines an
automorphism on the homotopy category of $A_\infty$-algebras.

\subsection{The opposite of an $A_\infty$-module}

Since modules over an $A_\infty$-algebra are defined using exactly
the same identities \textbf{SI(n)} as for $A_\infty$-algebras, and
since morphisms between modules satisfy only slight variants on the
identities \textbf{MI(n)}, essentially the same proofs show that the
opposite of a right $A$-module is a left $A^{\op}$-module, etc.  That
is, there are equivalences of categories
\begin{gather*}
(\text{left $A_\infty$-modules over $A$}) \xrightarrow{\op}
\Mod^{\infty}(A^{\op}), \\
\InfD(\text{left $A_\infty$-modules over $A$}) \xrightarrow{\op}
\InfD(A^{\op}).
\end{gather*}
So whenever left $A_\infty$-modules arise, we may easily convert
them to right $A_\infty$-modules, and vice versa.

\section{Adjunctions and equivalences}
\label{xxsec4}

This section lays more groundwork: generalities for establishing
equivalences between categories via Auslander and Bass classes,
results about derived functors for DG modules, and $\otimes$-Hom
adjointness.  Two of the main results of the section are Propositions
\ref{xxprop4.10} and \ref{xxprop4.11}, which describe when certain
subcategories of derived categories of DG-modules are equivalent.

\subsection{Auslander and Bass classes}

Let $\CC$ and $\DD$ be two categories. Let $F: \CC\to \DD$
be left adjoint to a functor $G: \DD\to \CC$. Then
there are natural transformations
\begin{gather*}
\eta : \id_{\CC} \rightarrow GF,\\
\varepsilon : FG \rightarrow \id_{\DD}.
\end{gather*}
We define two full subcategories as follows. The \emph{Auslander class}
associated to $(F,G)$ is the subcategory of
$\CC$ whose objects are
\[
\{M \,|\, \eta_M: M\to GF (M) \text{ is an isomorphism}  \}.
\]
The Auslander class is denoted by $\Aus$.
The \emph{Bass class} associated to $(F,G)$ is the subcategory
of $\DD$ whose objects are
\[
\{N \,|\, \epsilon_N: FG(N) \to N \text{ is an isomorphism} \}.
\]
The Bass class is denoted by $\Bass$. These definitions are
abstractions of ideas of Avramov and Foxby \cite[Section 3]{AF}.
The following lemma is proved by imitating \cite[Theorem 3.2]{AF}.

\begin{lemma}
\label{xxlem4.1}
Let $(F,G)$ be a pair of adjoint functors between $\CC$ and $\DD$.
\begin{enumerate}
\item
The functors $F$ and $G$ restrict to an equivalence of
categories between $\Aus$ and $\Bass$.
\item
If $\CC$ and $\DD$ are additive and $F$ and $G$ are additive
functors, then $\Aus$ and $\Bass$ are additive subcategories.
If $\CC$ and $\DD$ are triangulated and $F$ and $G$
are triangulated functors, then $\Aus$ and $\Bass$ are
triangulated subcategories.
\end{enumerate}
\end{lemma}

\subsection{Derived functors over a DG algebra}
\label{xxsec4.2}

The derived category and derived functors over a DG algebra are
well-understood constructions nowadays.  See \cite{Sp}, \cite{Ke94},
and \cite{FHT01}, for example.  We review some details in this
subsection.  As with $A_\infty$-algebras and modules, every DG module
in this paper is ${\mathbb Z}\times{\mathbb Z}$-graded.

Let $R$ be a DG algebra and let $M$ be a DG $R$-module. Then $M$ is
called \emph{acyclic} if $HM=0$; it is called \emph{free} if it is
isomorphic to a direct sum of shifts of $R$; and it is called
\emph{semifree} if there is a sequence of DG submodules
\[
0=M_{-1}\subset M_0\subset \cdots \subset M_n \subset \cdots
\]
such that $M=\bigcup_{n} M_n$ and that each $M_n/M_{n-1}$ is free
on a basis of cycles. Semifree modules are a replacement for
free complexes over an associative algebra.

\begin{notation}
If $R$ is a DG algebra and $M$ and $N$ are DG $R$-modules, we write
$\Hom_R(M,N)$ for the DG $k$-module whose degree $n$ elements are
degree $n$ $R$-module maps $M \rightarrow N$, ignoring the
differential; see Subsection~\ref{xxsec1.1} for the formula for the
differential in $\Hom_R(M,N)$.  Similarly, $\End_{R}(M)$ means the
complex $\Hom_R(M,M)$.
\end{notation}

In DG homological algebra, $K$-projective and $K$-injective DG modules
are used to define derived functors.  A DG $R$-module $M$ is called
\emph{$K$-projective} if the functor $\Hom_R(M,-)$ preserves
quasi-isomorphisms, or equivalently, $\Hom_R(M,-)$ maps acyclic DG
$R$-modules to acyclic DG $k$-modules. For example, a semifree DG
$R$-module is always $K$-projective.  A DG $R$-module $M$ is called
\emph{$K$-flat} if the functor $M\otimes_R -$ preserves
quasi-isomorphisms; every $K$-projective DG $R$-module is $K$-flat.  A
DG $R$-module $N$ is called \emph{$K$-injective} if the functor
$\Hom_R(-,N)$ preserves quasi-isomorphisms, or equivalently,
$\Hom_R(-,N)$ maps acyclic DG $R$-modules to acyclic DG $k$-modules.

Given a DG $R$-module $M$, a map $f: L \rightarrow M$ is called a
\emph{semifree} (or \emph{$K$-projective} or \emph{$K$-flat},
respectively) \emph{resolution} of $M$ if $f: L \rightarrow M$ is a
quasi-isomorphism and $L$ is semifree (or $K$-projective or $K$-flat,
respectively). Similarly, a \emph{$K$-injective} resolution of $M$ is a
quasi-isomorphism $M \rightarrow L$ where $L$ is $K$-injective.  In
all of these cases, we will also abuse notation slightly and refer to
$L$ itself as the resolution, omitting mention of the map $f$.

The right derived functor of $\Hom_R(M,N)$ is defined to be
\[
\RHom_R(M,N):=\Hom_R(P,N) \qquad {\text{or}}\qquad
\RHom_R(M,N):=\Hom_R(M,I)
\]
where $P$ is a $K$-projective resolution
of $M$ and $I$ is a $K$-injective resolution of $N$.
The left derived functor of $M\otimes_R N$ is defined to be
\[
M\otimes^L_R N:=M\otimes_R Q \qquad {\text{or}}\qquad
M\otimes^L_R N:=S\otimes_R N
\]
where $S$ is a $K$-flat resolution of $M$ and $Q$ is a
$K$-flat resolution of $N$.

\subsection{Tensor-Hom and Hom-Hom adjunctions}

We discuss $\otimes$-$\Hom$ and Hom-Hom adjointness, both basic and
derived.  These are well-known, at least in the case of modules over
an associative algebra.  The DG case may not be as familiar, so we
provide some details.  Here is the basic version.

\begin{lemma}
\label{xxlem4.2}
Let $A$, $B$, $C$, and $D$ be DG algebras.
\begin{enumerate}
\item \textup{[$\Hom$-$\Hom$ adjointness]}
Let $_AL_C$, $_DM_B$ and $_AN_B$ be DG bimodules.  Then
\[
\Hom_{A^{\op}}({_AL_C},\Hom_B({_DM_B},{_AN_B}))\cong
\Hom_B({_DM_B}, \Hom_{A^{\op}}({_AL_C},{_AN_B}))
\]
as DG $(C,D)$-bimodules.
\item \textup{[$\otimes$-$\Hom$ adjointness]}
Let $_DL_B$, $_BM_A$ and $_CN_A$ be DG bimodules.
Then
\[
\Hom_A({_D L}\otimes_B {M_A}, {_C N_A})\cong
\Hom_B({_DL_B}, \Hom_A({_BM_A}, {_CN_A}))
\]
as DG $(C,D)$-bimodules.
\end{enumerate}
\end{lemma}

The isomorphism in part (a) gives a pair of adjoint
functors
\[
A\text{-}\Mod \rightleftarrows (\Mod \,B)^{\op},
\]
namely $\Hom_{A^{\op}}(-,{_AN_B})$ (left adjoint) and
$\Hom_B(-,{_AN_B})$ (right adjoint).  This explains the label,
``$\Hom$-$\Hom$ adjointness.''

\begin{proof}
(a) The desired isomorphism
\[
\phi:
\Hom_{A^{\op}}({_AL_C},\Hom_B({_DM_B},{_AN_B}))\to
\Hom_B({_DM_B}, \Hom_{A^{\op}}({_AL_C},{_AN_B}))
\]
is defined by the following rule. Let
$f\in \Hom_{A^{\op}}({_AL_C},\Hom_B({_DM_B},{_AN_B}))$,
and $l\in L$ and $m\in M$, write
$f(l)\in \Hom_B({_DM_B},{_AN_B})$ and $f(l)(m)\in N$;
then $\phi(f): M\to \Hom_{A^{\op}}({_AL_C},{_AN_B})$ is
determined by
\[
\phi(f)(m)(l)=(-1)^{|l||m|} f(l)(m)
\]
for all $l\in L, m\in M$. It is straightforward to check that
$\phi$ is an isomorphism of DG bimodules.

The above construction is given in the unpublished manuscript
\cite[(3.4.2), p. 27]{AFH}. The isomorphism $\phi$ is called
the \emph{swap isomorphism} \cite[Sect. 3.4]{AFH}.

(b) This is standard and a proof is given in
\cite[(3.4.3), p. 28]{AFH}. A non-DG version is
in \cite[Theorem 2.11, p. 37]{Ro}.
\end{proof}

To get derived versions of these, we need information about bimodules,
semifree resolutions, $K$-projectives, etc.

\begin{lemma}
\label{xxlem4.3}
Let $A$ and $B$ be DG algebras. Let
$M$ and $L$ be DG $(B,A)$-bimodules, or equivalently,
DG $B^{\op}\otimes A$-modules.
\begin{enumerate}
\item
Let $N$ be a DG $A$-module. If there is a sequence
of DG submodules
\[
0=N_{-1}\subset N_0\subset \cdots \subset N_n \subset \cdots
\]
such that $N=\bigcup_n N_n$, $C_n:= N_{n}/N_{n-1}$
is $K$-projective, and the underlying graded module $C_n$
is projective, then $N$ is $K$-projective.
\item
If $M$ is semifree $(B,A)$-bimodule, then it is
$K$-projective over $A$. As a consequence, if $M\to L$ is a
semifree resolution of $L$, then restricted to the right-hand
side, it is a $K$-projective resolution of $L_A$.
\item
If $M$ is $K$-injective $(B,A)$-bimodule, then it is
$K$-injective over $A$. As a consequence, if $L\to M$
is a $K$-injective resolution of $L$, then restricted to the
right-hand side, it is a $K$-injective resolution of $L_A$.
\end{enumerate}
\end{lemma}

\begin{proof} (a) First consider the sequence
\[
0\to N_0\to N_1\to C_1\to 0
\]
of DG $A$-modules. This is a split (hence exact)
sequence after omitting the differentials, as
the underlying graded module $C_1$ is projective.
Let $X$ be an acyclic DG $A$-module. Then we have
an exact sequence
\[
0\to \Hom_A(C_1,X)\to \Hom_A(N_1,X)\to
\Hom_A(N_0,X)\to 0.
\]
If the two ends $\Hom_A(C_1,X)$ and $\Hom_A(N_0,X)$ are acyclic,
so is the middle term $\Hom_A(N_1,X)$. This shows that if
$C_0(=N_0)$ and $C_1$ are $K$-projective, so is $N_1$.
By induction on $n$ we see that $N_n$ is $K$-projective
and projective for all $n$. Since every
sequence
\[
0\to N_{n-1}\to N_n\to C_n\to 0
\]
splits, the map $\Hom_A(N_n,X)\to \Hom_A(N_{n-1},X)$ is
surjective. This means that the inverse system
$\{\Hom_A(N_n,X)\}_{n}$ satisfies Mittag-Leffler
condition. Since each $\Hom_A(N_n,X)$ is acyclic,
\[
\Hom_A(N,X)=\varprojlim \Hom_A(N_n,X)
\]
is acyclic by \cite[Theorem 3.5.8]{We}.

(b) The second assertion follows from the first one.

By part (a) and the definition of a semifree module,
we may assume $M$ is free. Since a free module is a
direct sum of shifts of $B^{\op}\otimes A$, we may
assume $M$ is a copy of $B^{\op}\otimes A$. We need to
show that $M$ is $K$-projective over $A$.

Let $N_A$ be a DG $A$-module that is acyclic. By
$\otimes$-$\Hom$ adjointness (Lemma \ref{xxlem4.2}(b)),
we have
\[
\Hom_A(B^{\op}\otimes A, N)\cong \Hom_k(B^{\op},
\Hom_A(A, N))\cong \Hom_k(B^{\op}, N).
\]
Since every DG $k$-module is $K$-projective,
$\Hom_k(B^{\op}, N)$ is acyclic.
Hence $B^{\op}\otimes A$ is $K$-projective over $A$.

(c) The second assertion follows from the first one.

Let $N_A$ be an acyclic DG $A$-module. By $\otimes$-$\Hom$
adjointness (Lemma \ref{xxlem4.2}(b)),  we have
\[
\Hom_A(N,{_BM_A})=\Hom_A(N, \Hom_{B^{\op}\otimes A}
(B^{\op}\otimes A, {_BM_A}))\cong\qquad\qquad\qquad\qquad
\]
\[
\qquad\qquad
\Hom_{B^{\op}\otimes A}(N\otimes_A (B^{\op}\otimes A), {_BM_A})
\cong \Hom_{B^{\op}\otimes A}(N\otimes B^{\op}, {_BM_A}).
\]
Since $N$ is acyclic, so is $N\otimes B^{\op}$. Since
$_BM_A$ is $K$-injective, the above formula implies that
$\Hom_A(N,{M_A})$ is acyclic. Hence $M_A$ is $K$-injective.
\end{proof}

We can combine the previous two lemmas to get derived $\Hom$-$\Hom$
and $\otimes$-$\Hom$ adjointness.

\begin{lemma}[Derived $\Hom$-$\Hom$ adjointness]
\label{xxlem4.4}
Let $A$, $B$, $C$, and $D$ be DG algebras. Let $_AL_B$, $_BM_C$
and $_AN_D$ be DG bimodules.
\begin{enumerate}
\item There is an isomorphism of complexes
\[
\RHom_{A^{\op}}({_AL_C},\RHom_{B}({_DM_B},{_AN_B}))\cong
\RHom_{B}({_DM_B}, \RHom_{A^{\op}}({_AL_C},{_AN_B}))
\]
in $\D(C^{\op}\otimes D)$.
\item There is an isomorphism of $k$-vector spaces
\[
\D(A^{\op})({_AL},\RHom_{B}({M_B},{_AN_B}))\cong
\D(B)({M_B}, \RHom_{A^{\op}}({_AL},{_AN_B})).
\]
\end{enumerate}
\end{lemma}

\begin{proof}
(a) This follows from Lemmas \ref{xxlem4.2} and \ref{xxlem4.3}, and by
taking semifree resolutions of the DG bimodules $M$ and $L$.

(b) This follows from (a) by taking $H^0$.
\end{proof}

\begin{lemma}[Derived $\otimes$-$\Hom$ adjointness]
\label{xxlem4.5}
Let $A$, $B$, $C$ and $D$ be DG algebras.
Let $_DL_B$, $_BM_A$ and $_CN_A$ be DG bimodules.
\begin{enumerate}
\item There is an isomorphism of complexes
\[
\RHom_A({_D L}\otimes^L_B {M_A}, {_C N_A})\cong
\RHom_{B}({_DL_B}, \RHom_A({_BM_A}, {_CN_A}))
\]
in $\D(C^{\op}\otimes D)$.
\item There is an isomorphism of $k$-vector spaces
\[
\D(A)({L}\otimes^L_B {M_A}, {N_A})\cong
\D(B)({L_B}, \RHom_A ({_BM_A}, {N_A})).
\]
\end{enumerate}
\end{lemma}

\begin{proof}
(a) This follows from Lemmas \ref{xxlem4.2} and \ref{xxlem4.3},
and by taking semifree resolutions of the
DG bimodules $M$ and $L$ and a $K$-injective resolution of
the bimodule $N$.

(b) This follows from (a) by taking $H^0$.
\end{proof}

\subsection{Balanced bimodules and equivalences}

The main results in this section are Propositions \ref{xxprop4.10} and
\ref{xxprop4.11}; these establish a framework for proving the derived
equivalences in Section \ref{xxsec5}.

Let $A$ and $E$ be two DG algebras. A DG $A$-module
$M$ is a DG $(E,A)$-bimodule if and only if
there is a map of DG algebras $E\to \Hom_A(M,M)$.

\begin{definition}
\label{xxdefn4.6}
Let $\BB$ be a DG $(E,A)$-bimodule. We call $\BB$ \emph{left balanced}
if there is a quasi-isomorphism
$\BB\to N$ of DG $(E,A)$-bimodules such that $N$ is
$K$-injective over $E^{\op}\otimes A$ and the
canonical map $E\to \Hom_A(N,N)$ is a quasi-isomorphism
of DG algebras. The \emph{right balanced} property is
defined in a similar way, in terms of the map $A^{\op}
\rightarrow \Hom_{E^{\op}}(N,N)$.
\end{definition}

\begin{lemma}
\label{xxlem4.7}
Let $\BB$ be a DG $(E,A)$-bimodule. Then
the following conditions are equivalent:
\begin{enumerate}
\item[(i)]
$\BB$ is left balanced.
\item[(ii)]
If $P\to \BB$ is a quasi-isomorphism of DG $(E,A)$-bimodules
with $P_A$ being $K$-projective, then the canonical map
$E\to \Hom_A(P,P)$ is a quasi-isomorphism of DG algebras.
\item[(iii)]
If $\BB\to I$ is a quasi-isomorphism of DG $(E,A)$-bimodules
with $I_A$ being $K$-injective, then the canonical map
$E\to \Hom_A(I,I)$ is a quasi-isomorphism of DG algebras.
\end{enumerate}
\end{lemma}

\begin{proof} (i) $\Leftrightarrow$ (ii).
Suppose that there is a quasi-isomorphism $\BB \to N$ where $N$ is
$K$-injective over $E^{\op} \otimes A$.  (By a result of Spaltenstein
\cite[Corollary 3.9]{Sp}, there always is such a map.)  By Lemma
\ref{xxlem4.3}(c), $N$ is $K$-injective over $A$. The
quasi-isomorphism $f: P\to \BB\to N$ induces two maps
\[
E\xrightarrow{i_N} \Hom_A(N,N)\xrightarrow{g} \Hom_A(P,N)
\]
and
\[
E\xrightarrow{i_P} \Hom_A(P,P)\xrightarrow{h} \Hom_A(P,N)
\]
of $(E,E)$-bimodules. Since $N$ is $K$-injective over
$A$ and $P$ is $K$-projective over $A$, both $g$ and $h$
are quasi-isomorphisms. It is easy to see that $gi_N=hi_P$: they
both map $e\in E$ to $ef=fe\in \Hom_A(P,N)$. Therefore
$i_N$ is a quasi-isomorphism if and only if $i_P$ is.

(ii) $\Leftrightarrow$ (iii).
This proof is similar.
\end{proof}

By the above lemma, we can construct plenty of left balanced
bimodules. For example, let $M$ be a $K$-injective (or $K$-projective)
DG $A$-module and let $E=\Hom_A(M,M)$. It follows from the lemma that
$M$ becomes a left balanced DG $(E,A)$-bimodule with its natural left
$E$-module structure.

We now recall a few definitions.

\begin{definition}
\label{xxdefn4.8}
An object $M$ in an additive category $\CC$ with infinite
direct sums is called \emph{small} if $\CC(M,-)$ commutes
with arbitrary direct sums.
\end{definition}

Let $A$ be an $A_\infty$-algebra and $M$ be a right
$A_\infty$-module over $A$.  Let $\trinf{A}{M}$ denote the
triangulated subcategory of $\InfD(A)$ generated by $M$ and its Adams
shifts. (Recall that every subcategory in this paper is full.)  Let
$\thinf{A}{M}$ be the thick subcategory generated by $\trinf{A}{M}$:
the smallest triangulated subcategory, closed under summands, which
contains $\trinf{A}{M}$.  Similarly, if $R$ is a DG algebra and $N$ is
a DG $R$-module, let $\triang{R}{N}$ denote the triangulated
subcategory of $\D (R)$ generated by $N$ and its shifts, and let
$\thick{R}{N}$ be the thick subcategory generated by $\triang{R}{N}$.
Let $\InfD_{\per}(A)=\thinf{A}{A}$, and let
$\D_{\per}(R)=\thick{R}{R}$.  We call objects in $\InfD_{\per}(A)$ and
$\D_{\per}(R)$ \emph{perfect complexes}.  Let $\locinf{A}{M}$ denote
the localizing subcategory (= triangulated and closed under arbitrary
direct sums) generated by $\trinf{A}{M}$, and similarly for
$\loc{R}{N}$.  We can also define each of these with $M$ or $N$
replaced by a collection of modules. It is well-known that $\loc{R}{R}
=\D(R)$.

\begin{lemma}
\label{xxlem4.9}
\begin{enumerate}
\item
If $R$ is a DG algebra and $N$ is a right DG $R$-module, then
$N$ is small in $\D(R)$ if and only if $\RHom_R(N,-)$ commutes
with arbitrary colimits, and if and only if $N\in \D_{\per}(R)$.
\item
If $A$ is an $A_\infty$-algebra and $M$ is a right
$A_\infty$-module over $A$, then $M$ is small in
$\InfD(A)$ if and only if $M\in \InfD_{\per}(A)$.
\end{enumerate}
\end{lemma}

\begin{proof} (a) The equivalence that $N$ is small if and only if $N$
is in $\D_{\per}(R)$ is somewhat standard; see Keller \cite[5.3]{Ke94},
for example.  If $\RHom_{R}(N,-)$ commutes with arbitrary colimits,
then since homology also commutes with colimits, $\D (R)(N,-) =
H\RHom_{R}(N,-)$ does as well, and so $N$ is small.  Finally,
$\RHom_{R}(R,-)$ commutes with arbitrary colimits, and hence so does
$\RHom_{R}(N,-)$ for any object $N$ in $\D_{\per}(R)=\thick{R}{R}$.

(b) This follows from part (a) and Propositions \ref{xxprop1.14}
and \ref{xxprop3.1}(c).
\end{proof}

Let $\BB$ be a DG $(E,A)$-bimodule. By Lemma
\ref{xxlem4.5}(b), $F_{\BB}:=- \ltensor{E} \BB: \D(E)\to \D(A)$
and $G_{\BB}:=\RHom_A(\BB, -):\D(A)\to \D(E)$ form a
pair of adjoint functors. Let $\ALOWB$ and $\BLOWB$ be the
Auslander and Bass classes associated to the
pair $(F_\BB,G_\BB)$. By Lemma \ref{xxlem4.1}, $(F_\BB,G_\BB)$
induces a triangulated equivalence between $\ALOWB$ and $\BLOWB$.

The next two results are precursors of the derived equivalences in the
next section.

\begin{proposition}
\label{xxprop4.10}
Let $A$ and $E$ be DG algebras, and
suppose that $\BB$ is a left balanced DG $(E,A)$-module.  Define
adjoint functors $F=F_{\BB}=-\ltensor{E}\BB$ and
$G=G_{\BB}=\RHom_A(\BB,-)$, as above.
\begin{enumerate}
\item
Then $F_{\BB}$ and $G_{\BB}$ induce an equivalence of categories
$\ALOWB\cong \BLOWB$.  Furthermore, $E_E\in \ALOWB$, $\BB_A\in
\BLOWB$, and $F_{\BB}(E_E)=\BB_A$.
\item
There is an equivalence of triangulated categories
\[
\triang{E}{E} \cong \triang{A}{\BB}.
\]
\item
There is an equivalence of triangulated categories
\[
\D_{\per}(E)=\thick{E}{E}\cong \thick{A}{\BB}.
\]
\item
If $\BB_A$ is small, then there is an equivalence of
triangulated categories
\[
 \D(E)\cong \loc{A}{\BB}.
\]
\end{enumerate}
\end{proposition}

\begin{proof} (a) The first assertion is Lemma \ref{xxlem4.1}.
Without loss of generality, we assume
that $\BB$ is $K$-injective over $E^{\op}\otimes A$.
Since $\BB$ is left balanced,
\[
E\to \Hom_A(\BB,\BB)\cong \RHom_A(\BB,\BB)
=\RHom_A(\BB, E\ltensor{E} \BB)=GF(E)
\]
is a quasi-isomorphism of DG $A$-modules. Hence
$E\in \ALOWB$. Clearly $F(E)=\BB$ and $G(\BB)=E$.
Consequently, $\BB\in \BLOWB$.

(b,c) These follow from (a).

(d) By definition, $F$ commutes with arbitrary colimits.
If $\BB$ is small, $G$ commutes with arbitrary colimits.
In this case, $\ALOWB$ and $\BLOWB$ have arbitrary colimits. Since $E\in
\ALOWB$, $\ALOWB=\D(E)$. Since $F$ is an equivalence,
$\BLOWB=\loc{A}{\BB}$.
\end{proof}

Now we consider two other functors $\tilde{F}^{\BB}=\RHom_{E^{\op}}(-,\BB):
\D(E^{\op})\to \D(A)$ and $\tilde{G}^{\BB}=\RHom_A(-,\BB):
\D(A)\to \D(E^{\op})$. Both of them are contravariant;
however, if we view them as $F^{\BB}: \D(E^{\op})\to \D(A)^{\op}$ and
$G^{\BB}: \D(A)^{\op}\to \D(E^{\op})$, then they become
covariant. By Lemma \ref{xxlem4.4}(a), $(F^{\BB},G^{\BB})$ is
an adjoint pair.  Let $\AUPB$ and $\BUPB$ be the
Auslander and Bass classes associated to the
pair $(F^{\BB},G^{\BB})$.

\begin{proposition}
\label{xxprop4.11}
Let $A$ and $E$ be DG algebras, and
suppose that $\BB$ is a left balanced DG $(E,A)$-module.  Define
adjoint functors $F=F^{\BB}$ and $G=G^{\BB}$, as above.
\begin{enumerate}
\item
Then $F^{\BB}$ and $G^{\BB}$ induce an equivalence of categories
$\AUPB\cong \BUPB$.  Furthermore,
$_EE\in \BUPB$, $\BB_A\in \AUPB$, and $F^{\BB}(_EE)=\BB_A$.
\item
If $\BB$ is also right balanced, then
$A_A\in \AUPB$, $_E\BB\in \BUPB$, and $F^{\BB}(_E\BB)=A_A$.
\item
There is an equivalence of triangulated categories
\[
\triang{E^{\op}}{E} \cong \triang{A}{\BB}^{\op}.
\]
If $\BB$ is also right balanced, then
\[
\triang{E^{\op}}{E,\BB} \cong \triang{A}{A,\BB}^{\op}.
\]
\item
There is an equivalence of triangulated categories
\[
\D_{\per}(E^{\op})=\thick{E^{\op}}{E}\cong \thick{A}{\BB}^{\op}.
\]
If $\BB$ is also right balanced, then
\[
\thick{E^{\op}}{E,\BB}\cong \thick{A}{A,\BB}^{\op}.
\]
\item
If $_E\BB$ is small and $\BB$ is right balanced, then there is
an equivalence of triangulated categories
\[
 \D_{\per}(E^{\op})\cong \thick{A}{A,\BB}^{\op}=
\thick{A}{\BB}^{\op}.
\]
As a consequence, $A\in \thick{A}{\BB}^{\op}$.
\end{enumerate}
\end{proposition}

\begin{proof} (a) The first assertion follows from
Lemma \ref{xxlem4.1}. We may assume that $\BB$ is $K$-injective
over $E^{\op}\otimes A$. Since $\BB$ is left balanced,
\[
G^{\BB}F^{\BB}(_EE)=\RHom_A(\RHom_{E^{\op}}(E,\BB),\BB)
=\RHom_A(\BB,\BB)\longleftarrow E
\]
is a quasi-isomorphism. This shows that $E\in \BUPB$.
Since $F^{\BB}(E)= \BB$, we have $\BB\in \AUPB$.

(b) This is the right-hand version of (a).

(c,d,e) These follow from (a,b).
\end{proof}

Proposition \ref{xxprop4.10} also implies the following easy
fact.

\begin{corollary}
\label{xxcor4.12} Let $M$ be an object in $\InfD(A)$ for some
$A_\infty$-algebra $A$. Then $\thick{A}{M}$ is triangulated
equivalent to $\D_{\per}(E)$ for some DG algebra $E$.
\end{corollary}

\begin{proof} By Proposition \ref{xxprop3.1} we may assume that $A$
is a DG algebra, and then we may replace $\InfD(A)$
by $\D(A)$. Hence we may assume that $M$ is a right DG $A$-module.

Let $\BB_A$ be a $K$-projective resolution of $M$ and
let $E=\End_A(\BB_A)$. Then $\BB$ is a left balanced
$(E,A)$-bimodule. Note that $M\cong \BB$ in $\D(A)$.
The assertion follows from Proposition \ref{xxprop4.10}(c).
\end{proof}

\section{Koszul equivalences and dualities}
\label{xxsec5}

In the setting of classical Koszul duality \cite{BGSo}, there is an
equivalence between certain subcategories of the derived categories of
a Koszul algebra $A$ and of its Koszul dual; the subcategories consist
of objects satisfying certain finiteness conditions.  In this section,
we explore the analogous results for non-Koszul algebras, DG algebras,
and $A_\infty$-algebras.  The main results are Theorems \ref{xxthm5.4}
and \ref{xxthm5.5} in the DG setting, and Theorems \ref{xxthm5.7} and
\ref{xxthm5.8} in the $A_\infty$ setting.

\subsection{Koszul equivalence and duality in the DG case}

Let $A$ be an augmented DG algebra and let $\koszul{A} =
(\baraug{A})^{\dual}$ be its Koszul dual, as defined in
Section~\ref{xxsec2.2}.  The usual bar construction $B(A;A)$
\cite[p. 269]{FHT01}, where the second $A$ is viewed as a DG
$A$-bimodule, agrees with the $A_\infty$-module version
$\baraug{(A;A)}$ from Section \ref{xxsec3.2}. By
\cite[Proposition 19.2(b)]{FHT01}, $B(A;A)$ is a semifree resolution of
the right $A$-module $k$. Thus to define derived functors we may
replace $k_A$ with $B(A;A)_A$.

The following lemma can be viewed as a dual version of
\cite[Proposition 19.2]{FHT01}.

\begin{lemma}
\label{xxlem5.1}
Let $\BB=B(A;A)$ and let $E=\koszul{A}$.
\begin{enumerate}
\item
The natural embedding $i: {_Ek}\to {_E\BB}$ is a quasi-isomorphism
of left DG $E$-modules.
\item
If $A$ is weakly Adams connected (Definition~\ref{xxdefn2.1}), then
$\BB$ is a $K$-injective DG left $E$-module.
\end{enumerate}
\end{lemma}

A left DG $E$-module is called \emph{semi-injective} if it is an
injective left graded $E$-module and a $K$-injective left DG
$E$-module.

\begin{proof}
(a) By \cite[Proposition 19.2(a)]{FHT01},
the augmentations in $BA$ and $A$ define a quasi-isomorphism
$\epsilon\otimes \epsilon:\BB_A\to k_A$ of right DG $A$-modules.
The map $i$ is a quasi-isomorphism
because the composition $k\xrightarrow{i} \BB
\xrightarrow{\epsilon\otimes \epsilon} k$ is the identity.
It is easy to see that the map $i$ is a left DG $E$-module
homomorphism.

(b) Let $E^{\dual}$ be the $E$-bimodule $\Hom_k(E,k)$. It follows
from the adjunction formula in Lemma \ref{xxlem4.2}(b) that
$E^{\dual}$ is semi-injective as a left and a right DG $E$-module.
If $V$ is a finite-dimensional $k$-vector space, then
$E^{\dual}\otimes V \cong \Hom_k(E,V)$ and this is also semi-injective
as a left and a right DG $E$-module.

Since $B(A;A)$ is locally finite, $E=(BA)^{\dual}$ is locally finite
and $E^{\dual}=(BA^{\dual})^{\dual} \cong BA$.  Hence $B(A;{_Ak})\cong
BA$ is a semi-injective left DG $E$-module.  By induction one can
easily show that if $M$ is a finite-dimensional left DG $A$-module,
then the bar construction $B(A;M)$ is a semi-injective left DG
$E$-module.  Since $A$ is weakly Adams connected, $A=\varprojlim N_n$
where the $\{N_n\}_{n\geq 0}$ are finite-dimensional left DG
$A$-modules.  Since each $N_n$ is finite-dimensional, we may further
assume that the map $N_n\to N_{n-1}$ is surjective for all $n$.  By
the assertion just proved, $B(A; N_n)$ is a semi-injective left DG
$E$-module for each $n$, as is $B(A; N_n/N_{n-1})$.

Since $A=\varprojlim N_n$ and since $B(A;A)$ is locally
finite, $B(A;A)=\varprojlim B(A;N_n)$.  A result of Spaltenstein
\cite[Corollary 2.5]{Sp} says that such an inverse limit of
$K$-injectives is again $K$-injective, and this finishes the proof.
(Spaltenstein's result is for inverse limits of $K$-injectives in the
category of chain complexes over an abelian category, but the proof is
formal enough that it extends to the category of DG modules over a DG
algebra.)
\end{proof}

\begin{remark}
\label{xxrem5.2} By the above lemma, $_E\BB$ is isomorphic to
$_E k$ in $\D(E^{\op})$. By \cite[Proposition 19.2(a)]{FHT01},
$\BB_A$ is isomorphic to $k_A$ in $\D(A)$. However, $\BB$ is
not isomorphic to $k$ in $\D(E^{\op}\otimes A)$ in general.
\end{remark}

\begin{lemma}
\label{xxlem5.3} Let $\BB$ be the right DG $A$-module $B(A;A)$
and let $C=\End_A(\BB)$.
\begin{enumerate}
\item
$\BB$ is a left balanced $(C,A)$-bimodule.
\item
If $E:=\koszul{A}$ is locally finite, then $\BB$
is a left balanced $(E,A)$-bimodule via the
natural isomorphism $E\to C$.
As a consequence, $H E\cong H\RHom_A(k,k)$.
\item
If $A$ is weakly Adams connected, then $\BB$ is a right balanced
$(E,A)$-bimodule.
\end{enumerate}
\end{lemma}

\begin{proof} (a) Since $\BB$ is semifree over $A$, the
assertion follows from Lemma \ref{xxlem4.7}.

(b) The first assertion follows from Lemma \ref{xxlem4.7} and the fact
that $E\to \End_A(\BB)$ is a quasi-isomorphism
\cite[Ex 4, p. 272]{FHT01}.

Since $\BB_A$ is a $K$-projective resolution of $k_A$,
\[
H E\cong H \End_A(\BB)\cong H\RHom_A(k,k).
\]

(c) By Lemma \ref{xxlem5.1}(b), $\BB$ is a $K$-injective left DG
$E$-module. To show that $\BB$ is right balanced, we must show that the
canonical map $\phi: A^{\op} \to \End_{E^{\op}}(\BB)$ is a
quasi-isomorphism.  This canonical map sends $a \in A$ to the
endomorphism $y \mapsto ya$. Since $_E\BB$ is $K$-injective,
$H\End_{E^{\op}}(\BB)\cong H\RHom_{E^{\op}}(k,k)$.  By part (b),
\[
H\RHom_{E^{\op}}(k,k)\cong H (\koszul{E})^{\op}
\cong H A^{\op},
\]
where the last isomorphism follows from Theorem
\ref{xxthm2.4}. Therefore, since $HA$ is locally finite,
it suffices to show that $H\phi$ is injective.
Let $a\in A^{\op}$ be a cocycle such that $\phi(a)=0$
in $H\End_{E^{\op}}(\BB)$. Then $a\neq 1$ and there is an
$f\in \End_{E^{\op}}(\BB)$ such that $\phi(a)=d(f)$.
Applying this equation to $x=[\; ] \otimes 1\in \BB=B(A;A)$,
we obtain
\begin{align*}
{[\; ] \otimes a} &=\phi(a)(x)=d(f)(x)\\
&=(d\circ f-(-1)^{\deg_1 f} f\circ d)(x)\\
&=d\circ f([\;]\otimes 1)\pm f\circ d([\;]\otimes 1).
\end{align*}
Since $f$ is a left $E$-module homomorphism,
$f([\;]\otimes 1)=[\;]\otimes w$ for some $w\in A$.
By definition, $d([\; ]\otimes 1) =0$. Therefore
\[
[\;]\otimes a =d\circ f([\;]\otimes 1)=d([\;]\otimes w)
=[\;]\otimes dw,
\]
and hence $a=dw$ as required.
\end{proof}

Here is a version of \cite[1.2.6]{BGSo} for DG algebras.

\begin{theorem}
\label{xxthm5.4}
Let $A$ be an augmented DG algebra and let $E=\koszul{A}$ be
the Koszul dual of $A$. Assume that $E$ is locally
finite.  The functors $\RHom_A(k,-)$ and $-\ltensor{E}k$ induce the
following equivalences.
\begin{enumerate}
\item
The category $\triang{A}{k}$ is triangulated equivalent to
$\triang{E}{E}$.
\item
The category $\thick{A}{k}$ is triangulated equivalent to
$\D_{\per}(E)$.
\item
Suppose that $k_A$ is small in $\D(A)$. Then $\loc{A}{k}$
is triangulated equivalent to $\D(E)$.
\end{enumerate}
\end{theorem}

\begin{proof} Note that $\BB=B(A;A)\cong k_A$ as a right
DG $A$-module. Then the assertions follow from Proposition
\ref{xxprop4.10} and Lemma \ref{xxlem5.3}(b).
\end{proof}

\begin{theorem}
\label{xxthm5.5}
Let $A$ be an augmented DG algebra and let $E=\koszul{A}$ be the
Koszul dual of $A$.  Assume that $A$ is weakly Adams connected.  The
functors $\RHom_A(-,k)$ and $\RHom_{E^{\op}}(-,k)$ induce the
following equivalences.
\begin{enumerate}
\item
The category $\triang{A}{k,A}^{\op}$ is triangulated equivalent to
$\triang{E^{\op}}{E,k}$.
\item
The category $\thick{A}{k,A}^{\op}$ is triangulated equivalent to
$\thick{E^{\op}}{E,k}$.
\item
Suppose that $k_A$ is small in $\D(A)$. Then
\[
\D_{\per}(A)^{\op}=\thick{A}{k,A}^{\op}\cong \thick{E^{\op}}{E,k}
=\thick{E^{\op}}{k}.
\]
\end{enumerate}
\end{theorem}

\begin{proof}
This follows from Proposition \ref{xxprop4.11} and Lemmas
\ref{xxlem5.1} and \ref{xxlem5.3}.
\end{proof}

\begin{corollary}
\label{xxcor5.6}
Let $A$ be a weakly Adams connected augmented DG algebra and let
$E=\koszul{A}$ be its Koszul dual.  If $k_A$ is small in $\D (A)$,
then $HE$ is finite-dimensional.
\end{corollary}

\begin{proof}
By Theorem \ref{xxthm5.5}(c), $E$ is in the thick subcategory
generated by $k$, and every object in $\thick{E^{\op}}{k}$ has
finite-dimensional homology.
\end{proof}

See Corollary \ref{xxcor6.2} for a related result for DG algebras, and
Corollaries \ref{xxcor5.9} and \ref{xxcor7.2} for similar results
about $A_\infty$-algebras.

\subsection{Koszul equivalence and duality in the $A_\infty$ case}

Now suppose that $A$ is an $A_\infty$-algebra.  By Proposition
\ref{xxprop1.14}, $A$ is quasi-isomorphic to the DG algebra $\env{A}$,
so by Proposition \ref{xxprop3.1}, the derived category $\InfD (A)$ is
equivalent to $\D (\env{A})$.  We can use this to prove the following,
which is a version of \cite[1.2.6]{BGSo} for
$A_\infty$-algebras.

\begin{theorem}
\label{xxthm5.7}
Let $A$ be an augmented $A_\infty$-algebra and let
$E=\koszul{A}$ be the Koszul dual of $A$. Assume that
$A$ is strongly locally finite (Definition~\ref{xxdefn2.1}).
\begin{enumerate}
\item
The category $\trinf{A}{k}$ is triangulated equivalent to
$\trinf{E}{E}$.
\item
The category $\thinf{A}{k}$ is triangulated equivalent to
$\InfD_{\per}(E)$.
\item
Suppose that $k_A$ is small in $\InfD(A)$.  Then
$\locinf{A}{k}$ is triangulated equivalent to $\InfD(E)$.
\end{enumerate}
\end{theorem}

\begin{proof} We can replace $A$ by $\env{A}$ and $E$ by $\koszul{\env{A}}
=\koszul{\koszul{\koszul{A}}}$. The assertions follow from Proposition
\ref{xxprop3.1} and Theorem \ref{xxthm5.4}.
\end{proof}

Similarly, Proposition \ref{xxprop3.1} combined with \ref{xxthm5.5}
give the following.

\begin{theorem}
\label{xxthm5.8}
Let $A$ be an augmented $A_\infty$-algebra and let
$E=\koszul{A}$ be the Koszul dual of $A$. Assume that
$A$ is strongly locally finite.
\begin{enumerate}
\item
The category $\trinf{A}{k,A}^{\op}$ is triangulated equivalent to
$\trinf{E^{\op}}{E,k}$.
\item
The category $\thinf{A}{k,A}^{\op}$ is triangulated equivalent to
$\thinf{E^{\op}}{E,k}$.
\item
Suppose that $k_A$ is small in $\InfD(A)$. Then
\[
\InfD_{\per}(A)^{\op}=\thinf{A}{k,A}^{\op}\cong
\thinf{E^{\op}}{E,k}
=\thinf{E^{\op}}{k}.
\]
\end{enumerate}
\end{theorem}

Just as Theorem \ref{xxthm5.5} implied Corollary \ref{xxcor5.6}, this
result implies the following.

\begin{corollary}
\label{xxcor5.9}
Let $A$ and $E$ be as in Theorem \ref{xxthm5.8}.
If $k_A$ is small in $\InfD (A)$, then $HE$ is finite-dimensional.
\end{corollary}

See Corollary \ref{xxcor7.2} for the converse of Corollary \ref{xxcor5.9}.

\begin{corollary}
\label{xxcor5.10}
Via the equivalence in Theorem~\ref{xxthm5.8}, there is an
isomorphism of $k$-vector spaces
\[
H^{i}_{j}(\RHom_A(k,A))\cong H^{-i}_{-j}(\RHom_{E^{\op}}(k,E))
\]
for all $i,j$.
\end{corollary}

\begin{proof}
Again we may assume that $A$ and $E$ are DG algebras.  The functor
$G^{\BB}$ in Proposition \ref{xxprop4.11} is defined as
$G^{\BB}=\RHom_A(-,\BB)\cong \RHom_A(-,k)$, and changes $S$ to
$S^{-1}$ and $\Sigma$ to $\Sigma^{-1}$; hence the assertion follows
from the fact $G^{\BB}(k_A)={_EE}$ and $G^{\BB}(A_A)={_Ek}$.
\end{proof}

\begin{remark}
One might hope to prove Theorems~\ref{xxthm5.7} and \ref{xxthm5.8}
directly, working in the category $\InfD (A)$ rather than $\D
(\env{A})$.  Keller \cite[6.3]{Ke01} and Lef\`evre-Hasegawa
\cite[4.1.1]{Le} have described the appropriate functors:
$A_\infty$-versions of $\RHom_A(-,-)$ and $- \ltensor{E} -$, which
they write as $\hominfty{A}(-,-)$ and $- \tensorinfty{E} -$.
Although we fully expect these to satisfy all of the required
properties (such as adjointness), it was easier to use the more
standard results in the DG setting.
\end{remark}

\section{Minimal semifree resolutions}
\label{xxsec6}

In this short section we consider the existence of a minimal semifree
resolution of a DG module over a DG algebra. The main result is
Theorem \ref{xxthm6.1}; this result is needed in several places.
There are similar results in the literature -- see, for example
\cite[Section 1.11]{AH} or \cite[Lemma A.3]{FHT88} -- but they require
that $A$ be connected graded with respect to the first (non-Adams)
grading.  We need to use this in other situations, though, so we
include a detailed statement and proof.

We say that $A$ is \emph{positively connected graded} in the second
(Adams) grading if $A^{\ast}_{<0} =0$ and $A^{\ast}_0=k$;
\emph{negatively connected graded} in the Adams grading
is defined similarly.
Write $\fm$ for the augmentation ideal of $A$; then a semifree
resolution $F\to M$ of a module $M$ is called \emph{minimal} if $d_F
(F) \subset F\fm$.

\begin{theorem}
\label{xxthm6.1}
Let $A$ be a DG algebra and let $M$ be a right DG $A$-module.
\begin{enumerate}
\item
Assume that $A$ is positively connected graded in the second grading
and that $HM^{\ast}_{\leq n}=0$ for some $n$, or that $A$ is
negatively connected graded in the second grading and that
$HM^{\ast}_{\geq n}=0$ for some $n$.  Then $M$ has a minimal semifree
resolution $L\to M$ with $L^{\ast}_{\leq n}=0$ (respectively
$L^{\ast}_{\geq n}=0$).
\item
Assume further that $HA$ and $HM$ are both bounded on the same side
with respect to the first grading: assume that for each $j$, there is
an $m$ so that $(HA)^{\leq m}_j=0$ and $(HM)^{\leq m}_j=0$, or
$(HA)^{\geq m}_j=0$ and $(HM)^{\geq m}_j=0$.
If $HA$ and $HM$ are locally finite (respectively, locally finite with
respect to the second grading), then so is $L$.
\end{enumerate}
\end{theorem}

\begin{proof}
Without loss of generality, we assume that $A^{\ast}_{< 0}=0$ and
$HM^{\ast}_{\leq n}=0$ for some $n$.
After an Adams shift we may further assume that $n=-1$, that is,
$HM=HM^{\ast}_{\geq 0}$. We will construct a sequence of
right DG $A$-modules $\{L_u\}_{u\geq 0}$
with the following properties:
\begin{enumerate}
\item[(1)]
$0=L_{-1}\subset L_0\subset \cdots \subset L_u\subset \cdots$,
\item[(2)]
$L_u/L_{u-1}$ is a free DG $A$-module generated by a cycles
of Adams-degree $u$,
\item[(3)]
$L_u\otimes_A k$ has a trivial differential; that is,
$d_{L_u} (L_u)\subset L_u\fm$,
\item[(4)]
There is a morphism of DG $A$-modules $\epsilon_u: L_u\to M$
such that the kernel and coker of $H(\epsilon_u)$
have Adams degree at least $u+1$.
\end{enumerate}
If $A$ and $M$ satisfy the hypotheses in part (b), then each $L_u$
will also satisfy
\begin{enumerate}
\item[(5)]
$L_u$ is locally finite (respectively, locally finite
with respect to the second grading).
\end{enumerate}

Let $L_{-1}=0$.  We proceed to construct $L_u$ inductively for $u \geq
0$, so suppose that $\{L_{-1}, L_0, \cdots, L_{u}\}$ have been
constructed and satisfy (1)--(4), and if relevant, (5).

Consider the map $H(\epsilon_u): HL_u\to HM$; let $C$ be its cokernel
and let $K$ be its kernel.  We will focus on the parts of these in
Adams degree $u+1$.  Choose an embedding $i$ of $C_{u+1}$ into the
cycles in $M$, and let $P_u$ be the free DG $A$-module $C_{u+1}
\otimes A$ on $C_{u+1}$, equipped with a map $f: P_u \rightarrow M$,
sending $x \otimes 1$ to $i(x)$ for each $x \in C_{u+1}$.  Since $A$
is positively connected graded in the Adams grading, the map $f$
induces an isomorphism in homology in Adams degrees up to and
including $u+1$.  Similarly, let $Q_u$ be the free DG $A$-module on
$K_{u+1}$, mapping to $L_u$ by a map $\tilde{g}$ inducing a homology
isomorphism in degrees less than or equal to $u+1$.  Then let
$L_{u+1}$ be the mapping cone of
\[
Q_u \xrightarrow{g} L_u \oplus P_u,
\]
where $g$ maps $Q_u$ to the first summand by the map $\tilde{g}$.
Since $Q_u$ is free and since the composite $(\epsilon_u + f) g$
induces the zero map on homology, this composite is
null-homotopic.  Therefore there is a map $\epsilon_{u+1}: L_{u+1}
\rightarrow M$.  In more detail, since $L_{u+1}$ is the mapping cone
of $g: Q_u\to J_u:=L_u \oplus P_u$, it may be written as
$L_{u+1}=S(Q_u)\oplus J_u$, with differential given by
\[
d(q, l)=(-d_{Q_u}(q), g(q)+d_{J_u}(l)).
\]
The null-homotopy gives an $A$-module homomorphism $\theta: Q_u\to M$
of degree $(-1,0)$ such that
\[
\theta d_{Q_u} +d_M \theta=\delta_u g
\]
where $\delta_u=\epsilon_u + f$.
The $A$-module homomorphism $\epsilon_{u+1}: L_{u+1}\to M$ is
defined by
\[
\epsilon_{u+1}(q,l)=\theta(q)+\delta_u(l)\quad \forall
q\in Q_u \text{ and }\; l\in J_u.
\]
One can check that $\epsilon_{u+1}$ commutes with the differentials
and hence is a morphism of DG $A$-modules. The morphism
$\epsilon_{u+1}$ is an extension of $\epsilon_u+f$, hence $\epsilon_u$
is the restriction of $\epsilon_{u+1}$ to $L_u$.

Now we claim that the kernel and cokernel of $H(\epsilon_{u+1})$ are in
Adams degree at least $u+2$.  There is a long exact sequence in
homology
\[
\cdots \rightarrow H^{i}_{j}(Q_{u}) \xrightarrow{H(g)} H^{i}_{j}(L_{u}\oplus
P_{u}) \rightarrow H^{i}_{j}(L_{u+1}) \xrightarrow{\delta}
H^{i+1}_{j}(Q_{u}) \rightarrow \cdots.
\]
In Adams degrees less than $u+1$, $Q_{u}$ and $P_{u}$ are zero, so
$L_{u}$ and $L_{u+1}$ are isomorphic.  In Adams degree $u+1$,
$H^{*}(Q_{u})$ maps isomorphically to $K_{u}$, a vector subspace of
$H(L_{u} \oplus P_{u})$, so the boundary map $\delta$ is zero, and the
above long exact sequence becomes short exact.  Indeed, there is a
commutative diagram, where the rows are short exact:
\[
\xymatrix{
0 \ar[r] & H^{i}(Q_u)_{\leq u+1} \ar[r] \ar[d]
& H^{i}(L_u \oplus P_u)_{\leq u+1} \ar[r] \ar[d]
& H^{i}(L_{u+1})_{\leq u+1} \ar[r] \ar[d]^{H(\epsilon_{u+1})} & 0\\
0 \ar[r] & 0 \ar[r] & H^{i}(M)_{\leq u+1}\ar[r] & H^{i}(M)_{\leq u+1}
\ar[r] & 0
}
\]
The snake lemma immediately shows that $H(\epsilon_{u+1})$ has zero
kernel and zero cokernel in Adams degree $\leq u+1$, as desired.  This
verifies property (4) for $L_{u+1}$.

With property (4), and the fact that $L_i/L_{i-1}$
has a basis of cycles in Adams degree $i$, (1) and (2) are easy to see.
To see (3) we use induction on $u$. It follows from
the construction and induction that $d_{L_u} (L_u)
\subset \fm L_u+ L_{u-1}$. Since the semibasis of $L_{u-1}$
has Adams degree no more than $u-1$ and the semibasis of
$P_u\oplus Q_u$ has Adams degree $u$, we see that
$d_{L_u} (L_u) \subset \fm L_u+ \fm L_{u-1}=\fm L_u$.

Let $L$ be the direct limit $\varinjlim L_u$. Then $L$ is semifree and
there is a map $\phi: L\to M$ such that the kernel and cokernel
of $H(\phi)$ are zero. Such an $L$ is a semifree
resolution of $M$. Property (3) implies that $L$ is minimal.

If the hypotheses of part (2) are satisfied, then
the construction of $L_u$ shows that (5) holds. Since
$L^{i}_{j}=(L_u)^{i}_{j}$ for $u\gg 0$, $L$ is also locally finite
(respectively, locally finite with respect to the Adams grading).
\end{proof}

We are often interested in the complex $\RHom_A(k,k)$ or in its
homology, namely, the $\Ext$-algebra, $\Ext_A^{\ast}(k,k)$.
As noted in Section \ref{xxsec4.2}, to compute
this, we replace $k$ by a $K$-projective resolution $P$, and then
$\RHom_A(k,k) = \Hom_A(P,k)$.  Since semifree implies $K$-projective,
we can use a minimal semifree resolution, as in the theorem.  The
construction of $L$ gives the following: for each $u$, there is a
short exact sequence of DG $A$-modules
\[
0 \rightarrow L_{u-1} \rightarrow L_u \rightarrow L_u / L_{u-1} \rightarrow 0.
\]
where $L_u/L_{u-1}$ is a free DG $A$-module, and this leads to a short
exact sequence
\[
0 \rightarrow \Hom_A(L_u/L_{u-1}, k) \rightarrow \Hom_A(L_u, k)
\rightarrow \Hom_A(L_{u-1}, k) \rightarrow 0.
\]
We see that $\Hom_A(L_u/L_{u-1}, k)_{i}=0$ when $i<u$, and therefore
\[
\Hom_A(L_u, k)_{u} \cong \Hom_A(L_u/L_{u-1},k)_{u} \cong (B^{\dual})_{-u},
\]
where $B$ is a graded basis for the free DG module $L_u/L_{u-1}$.
Note that $B$ is concentrated in degrees $(*,u)$, so its graded dual
$B^{\dual}$ is in degrees $(*,-u)$.
Again by the short exact sequence, by induction
on $u$, $\Hom_A(L_{u-1},k)_{i}=0$ when $i \geq u$, so we see that
\[
\RHom_A(k,k)_{u} \cong \Hom_A(L, k)_{u} \cong \Hom_A(L_u, k)_{u}
\cong (B^{\dual})_{-u}.
\]
Furthermore, since $L_{u}/L_{u-1}$ is free on a basis of cycles, or
alternatively because the resolution $L$ is minimal, we see that
\[
\Ext_A(k,k)_{u} \cong (B^{\dual})_{-u}.
\]
This leads to the following corollary; see Corollary \ref{xxcor5.6}
for a related result.

\begin{corollary}
\label{xxcor6.2}
Let $A$ be a DG algebra which is connected graded, either positively
or negatively, in the second grading, and let $E=\koszul{A}$ be its
Koszul dual.
\begin{enumerate}
\item Then $HE=\Ext_A^*(k,k)$ is finite-dimensional if and
only if $k_A$ is small in $\D(A)$.
\item If $HE$ is finite-dimensional (or if $k_A$ is small in $\D(A)$),
then $HA$ is Adams connected.
\end{enumerate}
\end{corollary}

\begin{proof}
(a)
If $HE=\Ext_A^*(k,k)$ is finite-dimensional, then by the above
computation, the minimal semifree resolution $L$ of $k$ is built from
finitely many free pieces, and so $L$ is a perfect complex: it is in
$\thick{A}{k}$.  Therefore $k_A$ is small in $\D(A)$.

Conversely, if $k_A$ is small, then it is isomorphic in $\D(A)$ to a
perfect complex, and we claim that if $X$ is a perfect complex, then
$H\RHom_A(X,k)$ is finite-dimensional as a vector space: this is true
if $X=A$, and therefore it is true for every object in the thick
subcategory generated by $A$.  Therefore $k_A$ small implies that $HE$
is finite-dimensional.

(b)
Now suppose that $HE$ is finite-dimensional.  Without loss of
generality, suppose that $A$ is positively graded connected in the
second grading.  We claim that for each $i$, $(HA)^{*}_{i}$ is
finite-dimensional.

Note that in the construction of the minimal semifree resolution for
$k$, the first term $L_0$ is equal to $A$, and the map $L_0
\rightarrow k$ is the augmentation.  Consider the short exact sequence
\[
0 \rightarrow L_{u-1} \rightarrow L_u \rightarrow L_u / L_{u-1}
\rightarrow 0,
\]
for $u \geq 1$.  Since $L_u/L_{u-1}$ is free on finitely many classes
in Adams degree $u$, then in Adams degree $i$, $H(L_u/L_{u-1})_{i}$ is
isomorphic to a finite sum of copies of $HA_{i-u}$.  Therefore if $i
\leq u$, then this is finite-dimensional.  Therefore when $i \leq u$,
$H(L_{u-1})_i$ is finite-dimensional if and only if $H(L_u)_{i}$ is.
For $i$ fixed and $u$ sufficiently large, $H(L_u)_i$ stabilizes and
gives $H(L)_i$.  But $HL \cong k$, since $L$ is a semifree resolution
of $k$.  Thus $H(L_0)_i=(HA)_i$ is finite-dimensional for each $i$, as
desired.
\end{proof}

\section{Towards classical Koszul duality}
\label{xxsec7}

In this section we recover the classical version of the Koszul duality
given by Beilinson-Ginzburg-Soergel \cite{BGSo}.  First we give some
useful results about $A_\infty$-algebras with finite-dimensional
cohomology, and then we use these to recover classical Koszul duality,
in Theorem \ref{xxthm7.5}.

\subsection{Finite-dimensional $A_\infty$-algebras}
\label{xxsec7.1}

Let $A$ be an $A_\infty$-algebra. Let $\InfD_{\fd}(A)$
denote the thick subcategory of $\InfD(A)$
generated by all $A_\infty$-modules $M$ over $A$
such that $HM$ is finite-dimensional.

\begin{lemma}
\label{xxlem7.1} Let $A$ be a strongly locally finite $A_\infty$-algebra.
For parts (b) and (c), assume that $A$ is Adams connected and that
$HA$ is finite-dimensional.
\begin{enumerate}
\item
$\thinf{A}{k}=\InfD_{\fd}(A)$.
\item
$A$ is quasi-isomorphic to a finite-dimensional Adams connected
DG algebra.
\item
$\thinf{A}{k}=\InfD_{\fd}(A)\supseteq \InfD_{\per}(A)$ and
$\locinf{A}{k}=\InfD(A)$.
\end{enumerate}
\end{lemma}

\begin{proof} (a) Clearly $\thinf{A}{k}\subset \InfD_{\fd}(A)$.  To
show the converse, we may replace $A$ by the DG algebra $\env{A} \cong
\koszul{\koszul{A}}$ (Proposition\ref{xxprop1.14}). Since $A$ is
strongly locally finite, so is $\koszul{\koszul{A}}$, by Lemma
\ref{xxlem2.2}.  So we may assume that $A$ is a strongly locally
finite DG algebra. In this case every 1-dimensional right DG
$A$-module $M$ is isomorphic to a shift of the trivial module $k$. As
a consequence, $M\in \thinf{A}{k}$. Induction shows that $M\in
\thinf{A}{k}$ if $M$ is finite-dimensional.  If $M$ is a right DG
$A$-module with $HM$ being finite-dimensional, then the minimal
semifree resolution $L$ of $M$ is Adams locally finite by Theorem
\ref{xxthm6.1}.  Thus $M$ is quasi-isomorphic to a finite-dimensional
right DG $A$-module by truncation: replace $L$ by $\bigoplus_{-N \leq
s \leq N} L^*_s$ for $N$ sufficiently large. Therefore $M$ is in
$\thinf{A}{k}$. This shows that $\thinf{A}{k}=\InfD_{\fd}(A)$.

(b) By Theorem \ref{xxthm2.4}, $A$ is
quasi-isomorphic to $B:=\koszul{\koszul{A}}$. Since $A$ is Adams connected,
so is $B$. Since $HA\cong HB$, $H(B^{\ast}_{\geq n})=0$ for some
$n$. Hence $B$ is quasi-isomorphic to $C:=B/B^{\ast}_{\geq n}$.
Therefore $A$ is quasi-isomorphic to the Adams connected
finite-dimensional DG algebra $C$.

(c) By part (b) we may assume that $A$ is finite-dimensional, which
implies that $A$ is in $\InfD_{\fd}(A)=\thinf{A}{k} \subseteq
\locinf{A}{k}$.  Therefore
\[
\locinf{A}{A} = \InfD (A) \subseteq \locinf{A}{k} \subseteq \InfD (A).
\]
This proves the last statement.
\end{proof}

\begin{corollary}
\label{xxcor7.2}
Let $A$ be an augmented $A_\infty$-algebra
and let $E=\koszul{A}$.
\begin{enumerate}
\item
Suppose that $A$ is strongly locally finite.
$k_A$ is small in $\InfD(A)$ if and only if $HE$ is
finite-dimensional.
\item
Suppose that $A$ is Adams connected.
If $k_A$ is small in $\InfD(A)$ (or if $HE$ is
finite-dimensional),
then $\InfD_{\per}(A)\cong \InfD_{\fd}(E)$.
\end{enumerate}
\end{corollary}

\begin{proof} (a) If $k_A$ is small, then
by Corollary \ref{xxcor5.9}, $HE$ is finite-dimensional.

Conversely, suppose that $HE$ is finite-dimensional.  By Proposition
\ref{xxprop1.14}, $A$ is quasi-isomorphic to the augmented DG algebra
$\env{A} \cong \koszul{\koszul{A}}$, so by Proposition
\ref{xxprop3.1}, $k_A$ is small in $\InfD(A)$ if and only if
$k_{\env{A}}$ is small in $\D(\env{A})$.  According to Corollary
\ref{xxcor6.2}, $k_{\env{A}}$ is small if and only if
$H\koszul{\env{A}}$ is finite-dimensional.  Since $A$ is strongly
locally finite, so is $E$, and we have a quasi-isomorphism
$\koszul{\koszul{E}} \cong E$.  This means that $H\koszul{\env{A}}
\cong H\koszul{A} = HE$.

(b) This follows from Theorem \ref{xxthm5.7}(b) (switching $A$ and
$E$) and Lemma \ref{xxlem7.1}(a) (for $E$).  (Since $A$ is Adams
connected, then by Lemma~\ref{xxlem2.2}, $E$ is Adams connected and
strongly locally finite, so these results hold for $E$.)
\end{proof}

Part (b) is Theorem \ref{xxthmB} from the introduction.

\subsection{Twisting the grading}
\label{xxsec7.2}

To recover the classical derived equivalences induced by
Koszul duality, we need to twist the grading of $\koszul{A}$.
Recall that if $A$ is Adams connected and is concentrated in degrees
$\{0 \} \times {\mathbb N}$,
then $\koszul{A}$ is in degrees $\{(n,-n)|n\geq 0\}$. The aim of this
section is to change the grading of $\koszul{A}$, to put it in degrees
$\{0\} \times {\mathbb N}$.

Let $A$ be a DG algebra with zero differential (there are
some problems if $d\neq 0$). Let $\overline{A}$ be the DG
algebra with zero differential such that
\begin{enumerate}
\item
$\overline{A}= A$ as ungraded associative algebras, and
\item
$\overline{A}^{i}_{j}=A^{i+j}_{-j}$ (so $A^{i}_{j}=\overline{A}^{i+j}_{-j}$).
\end{enumerate}
Note that the sign before $j$ is not essential; for example, we can
define a different $\overline{A}$ by
$\overline{A}^{i}_{j}=A^{i-j}_{j}$ or
$\overline{A}^{i}_{j}=A^{i-j}_{-j}$ or
$\overline{A}^{i}_{j}=A^{i+j}_{+j}$. Since $A$ is a DG algebra with
zero differential, so is $\overline{A}$.  If $A$ is Koszul, then
$A^!=H\overline{\koszul{A}}=\overline{H\koszul{A}}$ which lives in
degrees $\{0\} \times {\mathbb Z}$.

\begin{remark}
\label{xxrem7.3} If $A$ is a DG algebra with nonzero
differential $d$, then we don't know how to make
$\overline{A}$ into a DG algebra naturally. That is,
we don't see a good way of defining a differential
for $\overline{A}$.
\end{remark}

For any right DG $A$-module $M$, we define a corresponding right DG
$\overline{A}$-module $\overline{M}$ by
\begin{enumerate}
\item
$\overline{M}= M$ as ungraded right $A$-modules, and
\item
$\overline{M}^{i}_{j}=M^{i+j}_{-j}$.
\end{enumerate}
One can easily check that $\overline{M}$ with the differential
of $M$ is a right DG $\overline{A}$-module.

The following is routine.

\begin{lemma}
\label{xxlem7.4} The assignment $M\mapsto \overline{M}$
defines an equivalence between $\Mod \,A$ and $\Mod \,
\overline{A}$ which induces a triangulated equivalence between
$\D(A)$ and $\D(\overline{A})$. If further $A$ is
finite-dimensional, then $\D_{\fd}(A)\cong \D_{\fd}(\overline{A})$.
\end{lemma}

We now reprove \cite[Theorem 2.12.6]{BGSo} (in a slightly
more general setting).
Let $\InfD_{\fg}(A)$ denote the thick subcategory
of $\InfD(A)$ generated by all $A_\infty$-modules
$M$ over $A$ such that $HM$ is finitely generated
over $HA$. If $A$ is a DG algebra, $\D_{\fg}(A)$ is defined
in a similar way. Note that $\InfD_{\fg}(A)=\InfD_{\fd}(A)$
if and only if $HA$ is finite-dimensional.

\begin{theorem}
\label{xxthm7.5} Suppose $A$ is a connected graded
finite-dimensional Koszul algebra. Let $A^!=\overline{H\koszul{A}}$.
\begin{enumerate}
\item
$\D(A)\cong \loc{A^!}{k}$.
\item
If $A^!$ is right noetherian, then
$\D_{\fg}(A)=\D_{\fd}(A)\cong \D_{\per}(A^!)=\D_{\fg}(A^!)$.
\end{enumerate}
\end{theorem}

\begin{proof} Note that $A^!$ is quasi-isomorphic to $\overline{\koszul{A}}$.
It follows from Lemma \ref{xxlem7.4} that we can replace $A^!$ by
$E:=\koszul{A}$

(a) This follows from Theorem \ref{xxthm5.4}(b) (switching $A$ and $E$).

(b) By Corollary \ref{xxcor7.2}(b) (again
switching $A$ and $E$), $\D_{\per}(E)\cong \D_{\fd}(A)$.
Since $A$ is finite-dimensional, $\D_{\fd}(A)=\D_{\fg}(A)$.
Note that $A^!$ is concentrated in degrees $\{0\} \times {\mathbb N}$.
Since $A^!$ is right noetherian and has finite global dimension
(because $A$ is finite-dimensional Koszul),
it is clear that $\D_{\per}(A^!)=\D_{\fg}(A^!)$.
\end{proof}

\section{Some examples}
\label{xxsec8}

\begin{example}
\label{xxex8.1}
Let $\Lambda = \Lambda (x_{1}, x_{2})$ with $\deg x_{i}=(0,i)$.  As
noted in Example \ref{xxex2.8}, this is not a classical Koszul
algebra; nonetheless, it satisfies many of the same derived
equivalences as Koszul algebras.  Let $\Lambda^{!}$ denote the Ext
algebra $\Ext_{\Lambda}^{*}(k,k) \cong k[y_{1}, y_{2}]$, where $\deg
y_{i}=(1,-i)$; this is an $A_{\infty}$-algebra with no higher
multiplications.  Then $\koszul{\Lambda}$ is quasi-isomorphic to the
associative algebra $\Lambda^{!}$, so there is a triangulated
equivalence $\InfD (\koszul{\Lambda}) \cong \D (\Lambda^{!})$.
Similarly, we have $\InfD (\Lambda) \cong \D (\Lambda)$.  Since the
homology of $\koszul{\Lambda^{!}}  \cong \koszul{\koszul{\Lambda}}$ is
isomorphic to $\Lambda$, which is finite-dimensional, the trivial
module $k_{\Lambda^{!}}$ is small in $\D (\Lambda^{!})$.  So Theorems
\ref{xxthm5.7} and \ref{xxthm5.8} give triangulated equivalences
\begin{align*}
\thick{\Lambda}{k} &\cong \D_{\per}(\Lambda^{!}),
& \thick{\Lambda}{k,\Lambda}^{\op} &\cong
\thick{\Lambda^{!}}{\Lambda^{!}, k}, \\
\thick{\Lambda^{!}}{k} &\cong \D_{\per}(\Lambda),
& \D_{\per}(\Lambda^{!})^{\op} &\cong \thick{\Lambda}{k}, \\
\loc{\Lambda^{!}}{k} &\cong \D (\Lambda).
\end{align*}
Compare to the classical Koszul equivalences of Theorem
\ref{xxthm7.5}: part (b) of that theorem is the first of these
equivalences, while part (a) of the theorem is the last of these.

Slightly more generally, the results of Theorem \ref{xxthm7.5} hold
for exterior algebras on finitely many generators, as long as they are
graded so as to be Adams connected: $\Lambda (x_{1}, \dots, x_{n})$,
graded by setting $\deg x_{i}=(a_{i}, b_{i})$ with each $b_{i}$
positive, or each $b_{i}$ negative.
\end{example}

\begin{example}
\label{xxex8.2}
We fix an integer $p \geq 3$ and define two $A_\infty$-algebras, $B(0)$
and $B(p)$, each with $m_1=0$.  As associative algebras, they are both
isomorphic to $\Lambda (y) \otimes k[z]$ with $\deg y = (1,-1)$ and
$\deg z = (2,-p)$.  The algebra $B(0)$ has no higher multiplications,
while $B(p)$ has a single nonzero higher multiplication, $m_p$.  This
map $m_p$ satisfies $m_p(y^{\otimes p}) = z$; more generally, $m_p(a_1
\otimes \dots \otimes a_p)$ is zero unless each $a_i$ has the form
$a_i=yz^{j_i}$ for some $j_i \geq 0$, and
\[
m_p(yz^{j_1} \otimes \dots \otimes  yz^{j_p}) =
z^{1+\sum j_i}.
\]
See \cite[Example 3.5]{LP04} for more on $B(p)$.

We claim that $k_{B(0)}$ is not small in $\InfD (B(0)) \cong \D
(B(0))$, while $k_{B(p)}$ is small in $\InfD (B(p))$.  The Ext algebra
$\Ext_{B(0)}^*(k,k)$ is isomorphic to the associative algebra $k[u]
\otimes \Lambda (v)$ with $\deg u = (0,1)$ and $\deg v = (1,p)$.
In particular, this algebra is not finite-dimensional, so by Corollary
\ref{xxcor7.2}, $k_{B(0)}$ is not small in $\D (B(0))$.  In contrast,
by \cite[Proposition 12.6]{LP04}, the Koszul dual of $B(p)$ is
$A_{\infty}$-isomorphic to the associative algebra
$A(p)=k[x]/(x^{p})$, where $\deg x = (0,1)$.  Since $A(p)$ is
finite-dimensional, $k_{B(p)}$ is small in $\InfD (B(p))$.  This
verifies the claim.
\end{example}

\part{Applications in ring theory}

\section{The Artin-Schelter condition}
\label{xxsec9}

In this section we prove Corollaries \ref{xxcorD},
\ref{xxcorE} and \ref{xxcorF}.  We start by discussing Artin-Schelter
regularity, for both associative algebras and $A_\infty$-algebras.
The Eilenberg-Moore
spectral sequence is a useful tool for connecting results about
modules over $HA$ to modules over $A$, if $A$ is a DG algebra or an
$A_\infty$-algebra.  Then we discuss Frobenius algebras and prove
Corollary \ref{xxcorD}, and we discuss dualizing complexes and prove
Corollary \ref{xxcorE}.  At the end of the section, we prove
Corollary \ref{xxcorF}.

\subsection{Artin-Schelter regularity}
\label{xxsec9.1}

\begin{definition}
\label{xxdefn9.1}
Let $R$ be a connected graded algebra.
\begin{enumerate}
\item
$R$ is called \emph{Gorenstein} if $\injdim R_R=
\injdim {_RR}<\infty$.
\item
$R$ is called \emph{Artin-Schelter Gorenstein} if
$R$ is Gorenstein of injective dimension $d$ and
there is an integer $l$ such that
\[
\Ext^i_R(k,R)\cong \Ext^i_{R^{\op}}(k,R)\cong
\begin{cases} 0 & i\neq d,\\
\Sigma^l(k)& i=d.\end{cases}
\]
\item
$R$ is called \emph{Artin-Schelter regular} if
$R$ is Artin-Schelter Gorenstein and has global dimension $d$.
\end{enumerate}
\end{definition}

Artin-Schelter regular algebras have been used in many ways
in noncommutative algebraic geometry.

Now we consider analogues for $A_\infty$-algebras.  When $A$ is an
unbounded $A_\infty$-algebra, there is no good definition of global or
injective dimension, so we only consider a version of condition (b) in
Definition \ref{xxdefn9.1}.

\begin{definition}
\label{xxdefn9.2}
Let $A$ be an augmented $A_\infty$-algebra and let $k$ be the trivial
module.
\begin{enumerate}
\item
We say $A$ satisfies the \emph{right Artin-Schelter condition}
if there are integers $l$ and $d$ such that
\[
\Ext^i_A(k,A)\cong
\begin{cases} 0 & i\neq d,\\
\Sigma^l(k)& i=d.\end{cases}
\]
If further $k_A$ is small in $\InfD(A)$, then $A$ is called
\emph{right $A_\infty$-Artin-Schelter regular}, or just \emph{right
Artin-Schelter regular}, if the context is clear.
\item
We say $A$ satisfies the \emph{left Artin-Schelter condition}
if there are integers $l$ and $d$ such that
\[
\Ext^i_{A^{\op}}(k,A)\cong
\begin{cases} 0 & i\neq d,\\
\Sigma^l(k)& i=d.\end{cases}
\]
If further $_Ak$ is small in $\InfD(A^{\op})$, then $A$ is called
\emph{left $A_\infty$-Artin-Schelter regular} (or just \emph{left
Artin-Schelter regular}).
\item
We say $A$ satisfies the \emph{Artin-Schelter condition} if
the conditions in parts (a) and (b) hold for the same
pair of integers $(l,d)$. If further $k_A$ is small in $\InfD(A)$,
then $A$ is called \emph{$A_\infty$-Artin-Schelter regular} (or just
\emph{Artin-Schelter regular}).
\item
Finally, if $A$ is $A_\infty$-Artin-Schelter regular, we say that
$A$ is \emph{noetherian} if $\InfD_{\fg}(A) = \InfD_{\per}(A)$.
\end{enumerate}
\end{definition}

Suppose that $R$ is a connected graded algebra.  If $R$ is
Artin-Schelter regular as an associative algebra, then $R$ is
$A_\infty$-Artin-Schelter regular.  Conversely, if $R$ is
$A_{\infty}$-Artin-Schelter regular, then one can use Corollary
\ref{xxcor6.2} to show that $R$ is Artin-Schelter regular.

Also, if $R$ is a connected graded algebra, then
$\D_{\fg}(R)=\D_{\per}(R)$ if and only if $R$ is noetherian of finite
global dimension; this motivates our definition of ``noetherian'' for
$A_\infty$-Artin-Schelter regular algebras.

It is easy to see that $A$ satisfies the left
$A_\infty$-Artin-Schelter condition if and only if $A^{\op}$
satisfies the right $A_\infty$-Artin-Schelter condition. We
conjecture that the left $A_\infty$-Artin-Schelter condition is
equivalent to the right $A_\infty$-Artin-Schelter condition.  We verify
this, with some connectedness and finiteness assumptions, in Theorems
\ref{xxthm9.8} and \ref{xxthm9.11}.

\begin{proposition}
\label{xxprop9.3}
Let $A$ be an augmented $A_\infty$-algebra and let
$E=\koszul{A}$ be the Koszul dual of $A$.
\begin{enumerate}
\item
Assume that
$A$ is strongly locally finite. Then $A$ satisfies the right
Artin-Schelter condition if and only if $E$ satisfies
the left Artin-Schelter condition.
\item
Assume that $E$ is strongly locally finite.
Then $A$ satisfies the left
Artin-Schelter condition if and only if $E$ satisfies
the right Artin-Schelter condition.
\end{enumerate}
\end{proposition}

\begin{proof} (a) By Corollary \ref{xxcor5.10} $A$ satisfies
the right Artin-Schelter condition if and only if $E$
satisfies the left Artin-Schelter condition.

(b) Switch $A$ and $E$ and use the fact that $A \to \koszul{\koszul{A}}$
is a quasi-isomorphism: the assertion follows from part (a).
\end{proof}

\subsection{The Eilenberg-Moore spectral sequence}

In this subsection we recall the Eilenberg-Moore spectral sequence
\cite[Theorem III.4.7]{KM} and \cite[p. 280]{FHT01}.  This helps to
translate homological results for modules over $HA$ to similar
results for modules over $A$.

Suppose that $A$ is a DG algebra and that $M$ and $N$ are right DG
$A$-modules.  Because of the $\Z \times \Z$-grading on $A$, $M$, and
$N$, $\Ext_{HA}^*(HM,HN)$ is $\Z^{3}$-graded, and we incorporate the
gradings into the notation as follows:
\[
\Ext^p_{HA}(HM,HN)^{q}_{s}.
\]
On the other hand, $\Ext_A(M,N)$ is $\Z^{2}$-graded, since it is defined to
be the homology of $\RHom_A(M,N)$.  That is,
\[
\Ext^{i}_A(M,N)_{j} := \D (A) (M, S^{i} \Sigma^{j} N) \cong
H(\RHom_A(M, S^{i} \Sigma^{j} N)).
\]
Because of this, each $E_{r}$-page of the Eilenberg-Moore spectral
sequence is $\Z^{3}$-graded, while the abutment is $\Z^{2}$-graded.

\begin{theorem}
\label{xxthm9.4}
\cite[Theorem III.4.7]{KM}
Let $A$ be a DG algebra, and let $M$ and $N$ be right DG $A$-modules.
Then there is a spectral sequence of the form
\[
(E_2)^{p,q}_s \cong \Ext^p_{HA}(HM,HN)^q_s\Rightarrow \Ext^{p+q}_A(M,N)_s,
\]
natural in $M$ and $N$.  All differentials preserve the lower (Adams)
grading $s$.
\end{theorem}

This is a spectral sequence of cohomological type, with
differential in the $E_r$-page as follows:
\[
d_r: (E_r)^{p,q}_s \to (E_r)^{p+r, q-r+1}_s.
\]
Ignoring the Adams grading,
the $E_2$-term is concentrated in the right
half-plane (i.e., $p \geq 0$), and it converges strongly if
for each $(p,q)$, $d_{r}:E_{r}^{p,q} \rightarrow E_{r}^{p+r,q-r+1}$ is
nonzero for only finitely many values of $r$.

There is also a Tor version of this spectral sequence, which we do not
use.  See \cite[Theorem III.4.7]{KM} and \cite[p. 280]{FHT01} for more
details.

Note that the above theorem also holds for $A_\infty$-algebras;
see \cite[Theorem V.7.3]{KM}. Another way of obtaining an
$A_\infty$-version of this spectral sequence is
to use the derived equivalence between $\InfD(A)$ and $\D(\env{A})$.

\begin{corollary}
\label{xxcor9.5}
Let $A$ be an $A_\infty$-algebra. If $HA$ satisfies the
left (respectively, right) Artin-Schelter condition, then so does $A$.
\end{corollary}

\begin{proof} First of all we may assume that $A$ is a DG algebra.
Let $M=k$ and $N=A$. Note that $Hk=k$. Since $HA$ satisfies
the left Artin-Schelter condition, $\bigoplus_{p,q}
\Ext^p_{HA}(HM,HN)^q$ is 1-dimensional. By Theorem
\ref{xxthm9.4}, $\bigoplus_{n}\Ext_A^n(k,A)$ is 1-dimensional.
Therefore $A$ satisfies the left Artin-Schelter condition.
\end{proof}

One naive question is if the converse of Corollary \ref{xxcor9.5}
holds.

\subsection{Frobenius $A_\infty$-algebras}
\label{xxsec9.3}

In this subsection we define Frobenius DG algebras and
Frobenius $A_\infty$-algebras and then prove Corollary
\ref{xxcorD}.

\begin{definition}
\label{xxdefn9.6}
An augmented DG algebra $A$ is called \emph{left Frobenius}
(respectively, \emph{right Frobenius}) if
\begin{enumerate}
\item
$HA$ is finite-dimensional, and
\item
there is a
quasi-isomorphism of left (respectively, right) DG $A$-modules
$\alpha: S^l\Sigma^d(A) \to A^{\dual}$ for some integers $l$ and
$d$.
\end{enumerate}
An augmented DG algebra $A$ is called \emph{Frobenius}
if it is both left and right Frobenius.
\end{definition}

\begin{lemma}
\label{xxlem9.7}
Suppose $A$ is a DG algebra such that there is a
quasi-isomorphism of left DG $A$-modules
$\alpha: S^l\Sigma^d(A) \to A^{\dual}$ for some integers $l$ and
$d$.
\begin{enumerate}
\item
There is a
quasi-isomorphism of a right DG $A$-modules
$\beta: S^l\Sigma^d(A) \to A^{\dual}$ for the same integers $l$ and
$d$.
\item
If $A$ is connected graded with respect to some grading which is
compatible with the $\Z^{2}$-grading -- see below -- then $HA$ is
finite-dimensional.
\item
If $HA$ is finite-dimensional, then $HA$ is Frobenius as an
associative algebra.
\item
$A$ satisfies the left Artin-Schelter condition and
$\RHom_A(k,A)$ is quasi-isomorphic to $S^{-l}\Sigma^{-d}(k)$.
\end{enumerate}
\end{lemma}

The compatibility requirement for the grading in part (b) means
that there should be numbers $a$ and $b$ so that the $n$th graded
piece is equal to $\bigoplus_{ai+bj=n} A^i_j$.

\begin{proof} (a)
Let $\beta = S^{l}\Sigma^{d} (\alpha^{\dual})$.

(b) Since $S^l\Sigma^d(A)\to A^{\dual}$ is an quasi-isomorphism,
\[
(HA)^{n-l}_{m-d}\cong (HA^{\dual})^{-n}_{-m}\cong
((HA)^{-n}_{-m})^{\dual}
\]
for all $n,m$. This implies that $HA$ is locally finite.
If $A$ is connected graded with respect to some compatible grading,
then so is $HA$.  Then the above formula implies that $HA$ is
finite-dimensional.

(c) The quasi-isomorphism $\alpha$
gives rise to an isomorphism $H(\alpha): S^l\Sigma^d(HA)\to
(HA)^{\dual}$. If $HA$ is finite-dimensional, then $HA$ is
Frobenius.

(d) Since $A\to S^{-l}\Sigma^{-d}A^{\dual}$ is a
$K$-injective resolution of $A$, we can compute
$\RHom_A(k,A)$ by $\Hom_A(k,S^{-l}\Sigma^{-d}(A^{\dual}))$, which
is $S^{-l}\Sigma^{-d}(k)$.
\end{proof}

By Lemma \ref{xxlem9.7} above, $A$ is left Frobenius if and only
if it is right Frobenius. So we can omit both ``left'' and
``right'' before Frobenius. It is therefore easy to see that $A$ is Frobenius
if and only if $A^{\op}$ is. We show that the Frobenius
property is a homological property.

\begin{theorem}
\label{xxthm9.8}
Let $A$ be an Adams connected DG algebra such that $HA$ is
finite-dimensional. Then the following are equivalent.
\begin{enumerate}
\item
$A$ is Frobenius.
\item
$HA$ is Frobenius.
\item
$A$ satisfies the left Artin-Schelter condition.
\item
$A$ satisfies the right Artin-Schelter condition.
\end{enumerate}
\end{theorem}

\begin{proof} By Lemma \ref{xxlem7.1}, we may assume that
$A$ is finite-dimensional.

(a) $\Rightarrow$ (b): Use Lemma \ref{xxlem9.7}(c).

(b) $\Rightarrow$ (a): Since $HA$ is Frobenius,
there is an isomorphism of right $HA$-modules
$f: S^l\Sigma^d(HA)\to HA^{\dual}$. Pick $x\in ZA^{\dual}$ so
that the class of $x$ generates a submodule of $HA^{\dual}$
that is isomorphic to $S^l\Sigma^d(HA)$. Hence
the map $a\to xa$ is a quasi-isomorphism
$A\to A^{\dual}$.

(a) $\Rightarrow$ (d): By vector space duality,
\[
\RHom_{A^{\op}}(k,A)\cong \RHom_A(A^{\#},k)
\cong \RHom_A(S^l\Sigma^d(A),k)=S^{-l}\Sigma^{-d}(k).
\]
Hence $A$ satisfies the right Artin-Schelter condition.

(d) $\Rightarrow$ (a): Suppose $\RHom_{A^{\op}}(k,A)
\cong S^{-l}\Sigma^{-d}(k)$ by the right Artin-Schelter condition.
Since $HA$ is locally finite, by vector
space duality, we obtain that
$\RHom_A(A^{\#},k)$ is quasi-isomorphic to $S^{-l}\Sigma^{-d}(k)$.
By Theorem \ref{xxthm6.1} $A^{\dual}$ has a minimal
semifree resolution, say $P\to A^{\dual}$. Since $P$
is minimal, it has a semifree basis equal to
$(\bigoplus_{i}\Ext^i_A(A^{\#},k))^{\#}=\Sigma^d(k)$.
Hence $P\cong S^l\Sigma^d(A)$ for some $l$ and $d$.
Thus $S^l\Sigma^d(A)\to A^{\dual}$ is a quasi-isomorphism
and $A$ is left Frobenius. By Lemma \ref{xxlem9.7},
$A$ is Frobenius.

Thus we have proved that (a), (b) and (d) are equivalent.
By left-right symmetry, (a), (b) and (c) are equivalent.
\end{proof}

We obtain an immediate consequence.

\begin{corollary}
\label{xxcor9.9}
Let $f:A\to B$ be a quasi-isomorphism of Adams connected DG
algebras. Assume that $HA \cong HB$ is finite-dimensional.
Then $A$ is Frobenius if and only if $B$ is.
\end{corollary}

Since finite-dimensional Hopf algebras are Frobenius,
every finite-dimensional DG Hopf algebra is Frobenius.

Suggested by the DG case, we make the following definition.

\begin{definition}
\label{xxdefn9.10}
An $A_\infty$-algebra is called \emph{Frobenius} if $HA$ is
finite-dimensional and there is a quasi-isomorphism of right
$A_\infty$-modules $S^l\Sigma^d(A) \to A^{\dual}$ for some $l$
and $d$.
\end{definition}

The following is similar to Theorem \ref{xxthm9.8} and
the proof is omitted.

\begin{theorem}
\label{xxthm9.11}
Let $A$ be an Adams connected $A_\infty$-algebra such that
$HA$ is finite-dimensional. The following are equivalent.
\begin{enumerate}
\item
$A$ is Frobenius.
\item
$HA$ is Frobenius.
\item
$A$ satisfies the left Artin-Schelter condition.
\item
$A$ satisfies the right Artin-Schelter condition.
\end{enumerate}
\end{theorem}

Now we are ready to prove Corollary \ref{xxcorD} from the
introduction.  We restate it for the reader's convenience.

\setcounter{maintheorem}{3}
\begin{maincorollary}
Let $R$ be a connected graded algebra. Then $R$
is Artin-Schelter regular if and only if the {\rm Ext}-algebra
$\bigoplus_{i\in {\mathbb Z}}\Ext^i_R(k_R,k_R)$ is Frobenius.
\end{maincorollary}

\begin{proof}
Note that $R$ is right Artin-Schelter regular if and only if $k_R$ is
small and $R$ satisfies the right $A_\infty$-Artin-Schelter condition
\cite[Proposition 3.1]{SZ}.
Now we assume, temporarily, that $R$ is Adams connected; by Lemma
\ref{xxlem2.2}, this means that $\koszul{R}$ is, as well.  Then the
above conditions are in turn equivalent to: $H\koszul{R} =
\bigoplus_{i} \Ext^i(k_R, k_R)$ is finite-dimensional and $\koszul{R}$
satisfies the left Artin-Schelter condition.  By Theorem
\ref{xxthm9.11}, this is equivalent to $H\koszul{R}$ being Frobenius.

Now, we use Corollary \ref{xxcor6.2} to justify the Adams connected
assumption: if $R$ is Artin-Schelter regular, then $k_R$ is small,
which implies that $R=HR$ is Adams connected.  Conversely, if
$H\koszul{R}$ is Frobenius, then it is finite-dimensional, which also
implies that $R$ is Adams connected.
\end{proof}

This proof in fact shows that (for associative algebras)
left Artin-Shelter regularity is equivalent to right
Artin-Shelter regularity.

\subsection{Dualizing complexes and the Gorenstein property}
\label{xxsec9.4}

The balanced dualizing complex over a graded ring
$B$ was introduced by Yekutieli \cite{Ye}. We refer
to \cite{Ye} for the definition and basic properties.
Various noetherian graded rings have balanced dualizing
complexes; see \cite{Va,YZ}.

\begin{lemma}
\label{xxlem9.12}
Suppose $R$ is a noetherian connected graded ring with
a balanced dualizing complex. Then $R$ satisfies the right
Artin-Schelter condition if and only if $R$ is
(Artin-Schelter) Gorenstein.
\end{lemma}

\begin{proof} Let $B$ be the balanced dualizing complex
over $R$. Then the functor $F:=\RHom_{R}(-,B)$ induces
an equivalence $\D_{\fg}(R)\cong \D_{\fg}(R^{\op})^{\op}$
and satisfies $F(k_R)={_Rk}$.
By the right Artin-Schelter condition,
$\RHom_{R}(k,R)\cong S^l\Sigma^d(k)$ for some $l$ and $d$.
Applying the duality functor $F$, we have
\[
\RHom_{R^{\op}}(B,k)=
\RHom_{R^{\op}}(F(R),F(k))\cong \RHom_{R}(k,R)\cong
S^l\Sigma^d(k).
\]
Therefore $_RB$ is quasi-isomorphic to $S^{-l}\Sigma^{-d}(R)$.
Since $_RB$ has finite injective dimension by definition,
$_RR$ has finite injective dimension.  Also it follows from
$_RB\cong S^{-l}\Sigma^{-d}(R)$ that $B_R\cong S^{-l}\Sigma^{-d}(R)$
by the fact that $R^{\op}=\RHom_{R^{\op}}(B,B)$. So since $B_R$ has
finite injective dimension, so does $R_R$.

For the converse, one note that both $H\RHom_{R}(k,R)$
and $H\RHom_{R^{op}}(k,R)$ are finite-dimensional since
the existence of $B$ implies that $R$ satisfies the
$\chi$-condition. The assertion follows from the proof
of \cite[Lemma 1.1]{Zh}.
\end{proof}

Now we restate and prove Corollary \ref{xxcorE}.

\setcounter{maintheorem}{4}
\begin{maincorollary}
Let $R$ be a Koszul algebra and let $R^!$ be the
Koszul dual of $R$. If $R$ and $R^!$ are both noetherian having
balanced dualizing complexes, then $R$ is Gorenstein if and only if
$R^!$ is.
\end{maincorollary}

\begin{proof}
By Lemma \ref{xxlem9.12}
the Artin-Schelter condition is equivalent to the
Gorenstein property. The assertion follows from
Proposition \ref{xxprop9.3}.
\end{proof}

We say a connected graded algebra $A$ \emph{has enough normal
elements} if every nonsimple graded prime factor ring
$A/P$ contains a homogeneous normal element of positive
degree. A noetherian graded ring satisfying a polynomial
identity has enough normal elements.

\begin{corollary}
\label{xxcor9.13}
Let $R$ be a Koszul algebra and let $R^!$ the Koszul dual of $R$.
If $R$ and $R^!$ are both noetherian having enough normal elements,
then $R$ is (Artin-Schelter) Gorenstein if and only if $R^!$ is.
\end{corollary}

\begin{proof} By \cite[Theorem 5.13]{YZ}, $R$ and $R^!$ have
balanced dualizing complexes. By \cite[Proposition 2.3(2)]{Zh},
under the hypothesis, the Artin-Schelter Gorenstein property is
equivalent to the Gorenstein property. The assertion follows from
Corollary \ref{xxcorE}.
\end{proof}

Let $R^e=R\otimes R^{\op}$.
Following the work of Van den Bergh \cite{Va}, Ginzburg \cite{Gi}
and Etingof-Ginzburg \cite{EG}, an associative algebra $R$
is called {\it twisted Calabi-Yau} if
$$\Ext^i_{R^e}(R,R^e)\cong
\begin{cases} {^\phi R^1} &{\textrm{if }} i=d\\
0& {\textrm{if }} i\neq d\end{cases}$$
for some $d$ (note that we do not require $R$ to have finite
Hochschild dimension). If the above equation holds for
$\phi=Id_R$, then $A$ is called {\it Calabi-Yau}. If $R$
is connected graded, then ${^\phi R^1}$ should be replaced
by $\Sigma^l ({^\phi R^1})$ for some integer $l$.
It follows from Van den Bergh's result \cite[Proposition 8.2]{Va}
that if $R$ is connected graded noetherian and Artin-Schelter
Gorenstein, then $R$ is twisted Calabi-Yau. It is easy to see
that if $R$ has finite global dimension and $R^e$ is noetherian,
then Artin-Schelter regularity is equivalent to the
twisted Calabi-Yau property. It is conjectured that the
Artin-Schelter Gorenstein property is equivalent to the
twisted Calabi-Yau property for all connected graded noetherian
rings.

We end this section with a proof of Corollary \ref{xxcorF}.

\setcounter{maintheorem}{5}
\begin{maincorollary}
Let $A$ be an Adams connected commutative differential graded
algebra such that $\RHom_A(k,A)$ is not quasi-isomorphic to
zero. If the Ext-algebra $\bigoplus_{i\in {\mathbb Z}}
\Ext^i_A(k_A,k_A)$ is noetherian, then $A$ satisfies the
Artin-Schelter condition.
\end{maincorollary}

\begin{proof}
Since $A$ is commutative, its Ext algebra
$H=H\koszul{A}=\Ext_A^{*}(k,k)$ is a graded Hopf algebra which
is graded cocommutative \cite[p.545]{FHT91}. By the hypotheses,
$H$ is noetherian. Hence it satisfies \cite[(1.1)]{FHT91}.
Since $\RHom_A(k,A)\neq 0$, Corollary \ref{xxcor5.10} implies
that $\RHom_{E^{\op}}(k,E)\neq 0$ where $E=\koszul{A}$.
By Theorem \ref{xxthm9.4} $\RHom_{H^{\op}}(k,H)\neq 0$.
Since $H^{\op}\cong H$, we have $\RHom_{H}(k,H)\neq 0$
which says that $H$ has finite depth. By \cite[Theorem C]{FHT91},
the noetherian property of $H$ implies that $H$ is elliptic,
and elliptic Hopf algebras are classified in
\cite[Theorem B]{FHT91}. It is well-known that the Hopf
algebras in \cite[Theorem B]{FHT91} are Artin-Schelter
Gorenstein. By Corollary
\ref{xxcor9.5}, $\koszul{A}$ satisfies the Artin-Schelter condition,
and therefore by Proposition \ref{xxprop9.3}, $A$ does as well.
\end{proof}

\section{The BGG correspondence}
\label{xxsec10}

The classical Bern\v{s}te{\u\i}n-Gel'fand-Gel'fand (BGG)
correspondence states that the derived category of coherent sheaves
over ${\mathbb P}^n$ is equivalent to the stable derived category over
the exterior algebra of $(n+1)$-variables \cite[Theorem 2]{BGG}.  Some
generalizations of this were obtained by Baranovsky \cite{Ba},
He-Wu \cite{HW}, Mori \cite{Mo} and so on. In this section we
prove a version of the BGG correspondence in the $A_\infty$-algebra
setting, as a simple application of Koszul duality.

If $R$ is a right noetherian ring, then the stable
bounded derived category over $R$, denoted by
$\underline{\D^b_{\fg}}(R)$, is defined to be
the  Verdier quotient $\D_{\fg}(R)/\D_{\per}(R)$.
With $R$ concentrated in degrees $\{0\} \times \Z$, every
complex in $\D_{\fg}(R)$ is bounded. When $R$ is a
finite-dimensional Frobenius algebra, then the stable module
category over $R$ is equivalent to the stable bounded derived
category over $R$ \cite{Ri}.

Recall from Sections \ref{xxsec7.1} and \ref{xxsec7.2},
respectively, that
\begin{align*}
\InfD_{\fd}(A) &= \thinf{A}{M \,|\, M \in \Mod \,A, \, \dim_{k} HM< \infty}, \\
\InfD_{\fg}(A) &= \thinf{A}{M \,|\, M \in \Mod \,A, \, HM \ \text{finitely
generated over } HA}.
\end{align*}

If $HA$ is finite-dimensional, then by Lemma \ref{xxlem7.1}(a),
$\InfD_{\fd}(A)=\thinf{A}{k}$, which is also equal to
$\thinf{A}{k,A}$. Modelled on the definition of
$\underline{\D^b_{\fg}}(R)$, we define the \emph{stable
derived category} of an $A_\infty$-algebra $A$ to be
\[
\underline{\InfD_{\fg}}(A)=
\InfD_{\fg}(A)/\InfD_{\per}(A)
\]
where the right-hand side of the equation is a Verdier
quotient. Boundedness of complexes does not make sense here since
$A$ itself may not be bounded, but for a weakly Adams
connected $A_\infty$-algebra we have the following variation.
The \emph{small stable derived category} of a weakly
Adams connected $A_\infty$-algebra $A$ is defined
to be
\[
\underline{\InfD_{\sm}}(A)=
\thinf{A}{k,A}/\InfD_{\per}(A)
\]
where the right-hand side of the equation is a Verdier
quotient. If $HA$ is finite-dimensional, then
\[
\underline{\InfD_{\sm}}(A)=
\InfD_{\fd}(A)/\InfD_{\per}(A)=
\underline{\InfD_{\fg}}(A).
\]
In general, $\underline{\InfD_{\sm}}(A)\subset
\underline{\InfD_{\fg}}(A)$.
It is arguable whether or not
$\underline{\InfD_{\sm}}(A)$ is a good definition.
One reason we use the above definition is to make the BGG
correspondence easy to prove.

It is easy to see that $\underline{\InfD_{\sm}}(A)=0$ if and only if
$k_A$ is small (see Definition \ref{xxdefn4.8} and Lemma
\ref{xxlem4.9}(b)). In this case we call $A$ \emph{regular}.  This is
consistent with terminology for associative algebras: Orlov called the
triangulated category $\underline{\D}^b_{\fg}(R)$ the \emph{derived
category of the singularity of $R$} \cite{Or}.

Given any connected graded ring $R$, we define the
projective scheme over $R$ to be the quotient category
\[
\Proj R:=\Mod \,R/\cat{Tor} \,R
\]
where $\cat{Tor} \, R$ is the Serre localizing subcategory generated
by all finite-dimensional graded right $R$-modules \cite{AZ}.
When $R$ is right noetherian, we denote its noetherian
subcategory by $\proj R$. The bounded derived category of
$\proj R$ is $\D^b(\proj A)$ which is modelled by the
derived category of coherence sheaves over a projective
scheme. When $A$ is an $A_\infty$-algebra, we can define
the derived category directly without using $\Proj A$.
The \emph{derived category of projective schemes over $A$}
is defined to be
\[
\InfD(\Proj A)=\InfD(A)/\locinf{A}{k};
\]
the \emph{derived category of finite projective schemes
over $A$} is defined to be
\[
\InfD(\proj A)=\InfD_{\fg}(A)/\thinf{A}{k};
\]
and the \emph{derived category of small projective schemes
over $A$} is defined to be
\[
\InfD_{\sm}(\proj A)=\thinf{A}{k,A}/\thinf{A}{k}.
\]

If $A$ is right noetherian and regular (e.g., $A$ is a commutative
polynomial ring), then $\InfD(\proj A)=\InfD_{\sm}(\proj A)$ and
this is equivalent to the derived category of the (noncommutative)
projective scheme $\proj A$ \cite{AZ}.

\begin{lemma}
\label{xxlem10.1} Let $A$ be an $A_\infty$-algebra satisfying
the Artin-Schelter condition. Then
$\underline{\InfD_{\sm}}(A)^{\op}\cong
\underline{\InfD_{\sm}}(A^{\op})$.
\end{lemma}

\begin{proof} First of all we may assume $A$ is a DG algebra.
Let $\BB$ be the $(A,A)$-bimodule $A$. Clearly $\BB$ is a
balanced $(A,A)$-bimodule. The equivalences given in
Proposition \ref{xxprop4.10} are trivial, but
Proposition \ref{xxprop4.11} is not trivial.
By Proposition \ref{xxprop4.11}(a), $\AUPB\cong \BUPB$
and $F^\BB(_AA)=A_A\in \AUPB$. By the Artin-Schelter
condition, $F^{\BB}(_Ak)=\RHom_{A^{\op}}(k,A)=S^l\Sigma^d(k)$
and $G^{\BB}(S^l\Sigma^d(k))=S^{-l}\Sigma^{-d}(
G^{\BB}(k))=k$. Therefore $_Ak\in \AUPB$ and $F^{\BB}(_Ak)
=S^l\Sigma^d(k)$. This implies that $F^{\BB}$
induces an equivalence between $\thinf{A^{\op}}{k,A}^{\op}$
and $\thinf{A}{k,A}$ which sends $_AA$ to $A_A$.
The assertion follows from the definition of
$\underline{\InfD_{\sm}}(A)$.
\end{proof}

The next theorem is a version of the BGG correspondence.

\begin{theorem}
\label{xxthm10.2} Let $A$ be a strongly locally finite
$A_\infty$-algebra and let $E$ be its Koszul dual.
\begin{enumerate}
\item
There is an equivalence of triangulated categories
\[
\InfD_{\sm}(\proj A)^{\op}\cong
\underline{\InfD_{\sm}}(E^{\op}).
\]
\item
If $A$ is Adams connected graded noetherian right Artin-Schelter
regular, then there is an equivalence of triangulated
categories
\[
\InfD(\proj A)\cong \underline{\InfD_{\fg}}(E).
\]
\end{enumerate}
\end{theorem}

\begin{proof} (a) By Theorem \ref{xxthm5.8}(b) there is an
equivalence of triangulated categories
\[
\thinf{A}{k,A}^{\op}\cong \thinf{E^{\op}}{k,E}
\]
which maps $k_A$ to $E$. Therefore we have
\[
\begin{aligned}
\InfD_{\sm}(\proj A)^{\op}&
=(\thinf{A}{k,A}/\thinf{A}{k})^{\op}\\
&\cong \thinf{E^{\op}}{k,E}/
\thinf{E^{\op}}{E}=\underline{\InfD_{\sm}}(E^{\op}).
\end{aligned}
\]

(b) Since $A$ is AS regular, Corollary \ref{xxcorD} says that $E$ is
Frobenius and hence $HE$ is finite-dimensional. By Lemma \ref{xxlem10.1},
$\underline{\InfD_{\sm}}(E^{\op})\cong \underline{\InfD_{\sm}}(E)^{\op}$.
Since $HE$ is finite-dimensional,
$\underline{\InfD_{\sm}}(E)=\underline{\InfD_{\fg}}(E)$.
Since $A$ is connected graded Artin-Schelter regular, $k_A$ is
small, so
\[
\thinf{A}{k,A} = \thinf{A}{A} = \InfD_{\per}(A).
\]
By the definition of noetherian (Definition \ref{xxdefn9.2}),
$\InfD_{\fg}(A)=\InfD_{\fg}(A)$.  Therefore $\InfD(\proj A) =
\InfD_{\sm}(\proj A)$.  The assertion follows from part (a).
\end{proof}

Theorem \ref{xxthmC} is part (b) of the above theorem.

\section*{Acknowledgments}

D.-M. Lu is supported by the NSFC (project 10571152) of
China. Q.-S. Wu is supported by the NSFC (key project
10331030) of China, Doctorate Foundation (No. 20060246003),
Ministry of Education of China. J.J. Zhang is supported by
the NSF of USA and the Royalty Research Fund of the University
of Washington.

\def\cftil#1{\ifmmode\setbox7\hbox{$\accent"5E#1$}\else
  \setbox7\hbox{\accent"5E#1}\penalty 10000\relax\fi\raise 1\ht7
  \hbox{\lower1.15ex\hbox to 1\wd7{\hss\accent"7E\hss}}\penalty 10000
  \hskip-1\wd7\penalty 10000\box7}
\providecommand{\bysame}{\leavevmode\hbox to3em{\hrulefill}\thinspace}
\providecommand{\MR}{\relax\ifhmode\unskip\space\fi MR }
\providecommand{\MRhref}[2]{%
  \href{http://www.ams.org/mathscinet-getitem?mr=#1}{#2}
}
\providecommand{\href}[2]{#2}

\end{document}